\documentclass[a4paper,notitlepage,twoside,leqno,12pt]{amsart}

\usepackage{mathrsfs}
\usepackage{bbm,pifont,latexsym}
\usepackage{anysize}  
\marginsize{2.5cm}{2.5cm}{2.0cm}{2.0cm}

\usepackage{dcolumn,indentfirst,color}
\usepackage[pdfpagemode=UseNone,colorlinks=true,linkcolor=black,citecolor=black,pdfstartview=FitH]{hyperref}
\usepackage{amsmath,amssymb,amscd,amsthm,amsfonts,mathrsfs}
\usepackage{color,graphicx,xcolor,graphics}
\usepackage{titlesec} 
\usepackage{titletoc} 
\titlecontents{section}[20pt]{\addvspace{2pt}\filright}
              {\contentspush{\thecontentslabel. }}
              {}{\titlerule*[8pt]{}\contentspage}

\usepackage{fancyhdr} 
\newtheorem{thm}{Theorem}[section]
\newtheorem{lema}[thm]{Lemma}
\newtheorem{cor}[thm]{Corollary}
\newtheorem{prop}[thm]{Proposition}

\theoremstyle{definition}
\newtheorem{defi}[thm]{Definition} 
\newtheorem{rmk}[thm]{Remark}

\newcommand{\D}{\mathbb{D}}
\newcommand{\T}{\mathbb{T}}
\newcommand{\R}{\mathbb{R}}
\newcommand{\Z}{\mathbb{Z}}

\newcommand{\C}{\mathbb{C}}
\newcommand{\EC}{\overline{\mathbb{C}}}
\newcommand{\MM}{\mathcal{M}}
\newcommand{\A}{\mathcal{A}}
\newcommand{\MH}{\mathcal{H}}
\newcommand{\MO}{\mathcal{O}}
\newcommand{\Crit}{\textup{Crit}}
\newcommand{\Teich}{\textup{Teich}}
\newcommand{\Mod}{\textup{Mod}}

\makeatletter\@addtoreset{equation}{section}\makeatother 

\titleformat{\section}{\centering\normalsize}{\textsc{\thesection.}}{1em}{\textsc}
\titleformat{\subsection}{\normalsize}{\thesubsection.}{1em}{\textbf}

\begin{document}

\author{FEI YANG}
\address{Fei Yang, Departments of Mathematics, Nanjing University, Nanjing, 210093, P. R. China}
\email{yangfei\rule[-2pt]{0.2cm}{0.5pt}math@163.com}

\author{JINSONG ZENG}
\address{Jinsong Zeng, School of Mathematical Sciences, Fudan University, Shanghai, 200433, P. R. China}
\email{10110180006@fudan.edu.cn}

\title{ON THE DYNAMICS OF A FAMILY OF GENERATED RENORMALIZATION TRANSFORMATIONS}

\begin{abstract}
We study the family of renormalization transformations of the generalized $d$--dimensional diamond hierarchical Potts model in statistical mechanic and prove that their Julia sets and non-escaping loci are always connected, where $d\geq 2$. In particular, we prove that their Julia sets can never be a Sierpi\'{n}ski carpet if the parameter is real. We show that the Julia set is a quasicircle if and only if the parameter lies in the unbounded capture domain of these models. Moreover, the asymptotic formula of the Hausdorff dimension of the Julia set is calculated as the parameter tends to infinity.
\end{abstract}



\date{\today}



\maketitle


\section{Introduction}

The statistical mechanical models on hierarchical lattices have attracted many interests recently since they exhibit a deep connection between their limiting sets of the zeros of the partition functions and the Julia sets of rational maps in complex dynamics \cite{BL,BLR1,BLR2,DDI,Qi,QL,QYG}. A celebrated Lee-Yang theorem \cite{LY,YL} in statistical mechanics asserts that the zeros of the partition function for some magnetic materials lie on the unit circle in the complex plane, which is corresponding to a purely imaginary magnetic field. This means that the complex singularities of the free energy lie on this line, where the free energy is the logarithm of the partition function.

The partition function $\textup{\textsf{Z}}=\textup{\textsf{Z}}(z,\textup{t})$ can be written as a Laurent polynomial in two variables $z$ and $\textup{t}$, where $z$ is a `field-like' variable and $\textup{t}$ is `temperature-like'. Note that the complex zeros of $\textup{\textsf{Z}}(z,\textup{t})$ in $z$ are called the \textit{Lee-Yang zeros} for a fixed $\textup{t}\in[0,1]$. Naturally, one can study the zeros of $\textup{\textsf{Z}}(z,\textup{t})$ in the $\textup{t}$-variable. These zeros are called \textit{Fisher zeros} since they were first studied by Fisher for regular two-dimensional lattice \cite{Fi,BK}. However, compared with the Lee-Yang zeros, Fisher zeros do not lie on the unit circle any more. For example, for the regular two-dimensional lattice, the Fisher zeros lie on the union of two circles $|\textup{t}\pm 1|=\sqrt{2}$. For more comprehensive introduction on Lee-Yang zeros and Fisher zeros, see \cite{BLR2} and the references therein.

In 1983, Derrida, de Seze and Itzykson showed that the Fisher circles of the Ising model on the regular two-dimensional lattice $\mathbb{Z}^2$ become a fractal Julia set upon replacing $\mathbb{Z}^2$ by a hierarchical lattice \cite{DDI}. They proved that the corresponding singularities of the free energy lie on the Julia set of the rational map
\begin{equation}
z\mapsto \left(\frac{z^2+\lambda-1}{2z+\lambda-2}\right)^2.
\end{equation}
This means that the distribution of the singularities of the free energy can have a pretty wild geometry. Henceforth, a lot of works related on the Julia sets of this renormalization transformation appeared (see \cite{AY,BL,Ga,HL,Qi,QL,QYG,WQYQG} and references therein). For the ideas formulated in renormalization transformation in statistical mechanics, see \cite{Wi}.

Recently, Qiao considered the generalized diamond hierarchical Potts model and proved that the family of rational maps
\begin{equation}\label{Umn}
U_{mn\lambda}(z) =\left(\frac{(z+\lambda-1)^m+(\lambda-1)(z-1)^m}{(z+\lambda-1)^m-(z-1)^m}\right)^n
\end{equation}
are actually the renormalization transformation of the \emph{generalized diamond hierarchical Potts model} \cite[Theorem 1.1]{Qi}, where $m,n\geq 2$ are both integers and $\lambda\in\C^*:=\C\setminus\{0\}$ is a complex parameter.
The standard diamond lattice ($m = n = 2$) and the \emph{diamond-like lattice} ($m = 2$ and
$n \in \mathbb{N}$) are the special cases of \eqref{Umn}.

In this paper, we will consider the case for $d:=m = n\geq 2$. For simplicity, we use $U_{d\lambda}$ to denote $U_{dd\lambda}$ in \eqref{Umn}. We not only study the topological properties of the Julia sets of $U_{d\lambda}$, but also consider the connectivity of the non-escaping locus of the parameter space of this renormalization transformation.

If $\lambda=0$, then $U_{d\lambda}$ degenerates to a parabolic polynomial $U_{d0}(z)=(\frac{z+d-1}{d})^d$ whose Julia set is a Jordan curve. For the connectivity of the Julia sets of $U_{d\lambda}$, we have the following theorem.

\begin{thm}\label{J-connected}
The Julia set of $U_{d\lambda}$ is always connected for every $d\geq 2$ and $\lambda\in\C^*$.
\end{thm}

Note that Qiao and Li proved that the Julia set of $U_{d\lambda}$ is connected for $d= 2$ and $\lambda\in \R$ \cite{QL}. We would like to remark that if $m\neq n$, then there exists parameter $\lambda\in\C^*$ such the Julia set of $U_{mn\lambda}$ defined in \eqref{Umn} is disconnected (see \cite[Figure 3.1]{Qi} for example).

Let $\overline{\mathbb{C}}=\mathbb{C}\cup\{0\}$ be the Riemann sphere. According to \cite{Wh}, a connected and locally connected compact set $S$ in $\EC$ is called a \emph{Sierpi\'{n}ski carpet} if it has empty interior and can be written as $S=\overline{\mathbb{C}}\setminus \bigcup_{i\in\mathbb{N}}D_i$, where $\{D_i\}_{i\in\mathbb{N}}$ are Jordan regions satisfying $\partial D_i\cap\partial D_j=\emptyset$ for $i\neq j$ and the spherical diameter $\text{diam}(\partial D_i)\rightarrow 0$ as $i\rightarrow \infty$.

The first example of the Sierpi\'{n}ski carpet as the Julia set of a rational map was given in \cite[Appendix F]{Mi1}. Afterwards, many families of the rational maps serve the examples such that their Julia sets are Sierpi\'{n}ski carpets for suitable parameters. See \cite{DLU} for the family of McMullen maps and \cite{XQY} for generated McMullen maps. However, for the  renormalization transformation $U_{d\lambda}$, we have the following theorem.

\begin{thm}\label{thm-no-sier}
For $d\geq 2$ and $\lambda\in\R$, the Julia set of $U_{d\lambda}$ is not a Sierpi\'{n}ski carpet.
\end{thm}

The proof of Theorem \ref{thm-no-sier} is based on proving the intersection of the boundaries of two of the Fatou components of $U_{d\lambda}$ are always non-empty (see Lemma \ref{lema-non-empty} and Theorem \ref{thm-no-sier-resta}).

The \emph{Mandelbrot set} of quadratic polynomials $f_c(z)=z^2+c$ is defined by
\begin{equation*}
M=\{c\in\C: f_c^{\circ n}(0)\not\rightarrow\infty \text{~as~} n\rightarrow\infty\}.
\end{equation*}
Douady and Hubbard showed that $M$ is connected \cite{DH}. For higher degree polynomials with only one critical point, there are associated \emph{Multibrot sets}.
For rational maps, one way to study the parameter space is to consider the \emph{connectedness locus}, which consists of all parameters such the corresponding Julia set is connected. However, the connectedness locus makes no sense in our case since every Julia set is connected.

For $\lambda\neq 0$, then 1 and $\infty$ are two superattracting fixed points of $U_{d\lambda}$. The \emph{non-escaping locus} $\MM_d$ associated to this family is defined by
\begin{equation}\label{Mandel}
\MM_d=\{\lambda\in\C^*:U_{d\lambda}^{\circ n}(0)\not\rightarrow 1 \text{~and~} U_{d\lambda}^{\circ n}(0)\not\rightarrow \infty \text{~as~}n\rightarrow\infty\}\cup\{0\}.
\end{equation}
Obviously, ``non-escaping" here means the collection of those parameters such that the orbit of $0$ cannot be attracted by $1$ and $\infty$. Note that $0$ is a critical value of $U_{d\lambda}$.

\begin{figure}[!htpb]
  \setlength{\unitlength}{1mm}
  \centering
  \includegraphics[width=70mm]{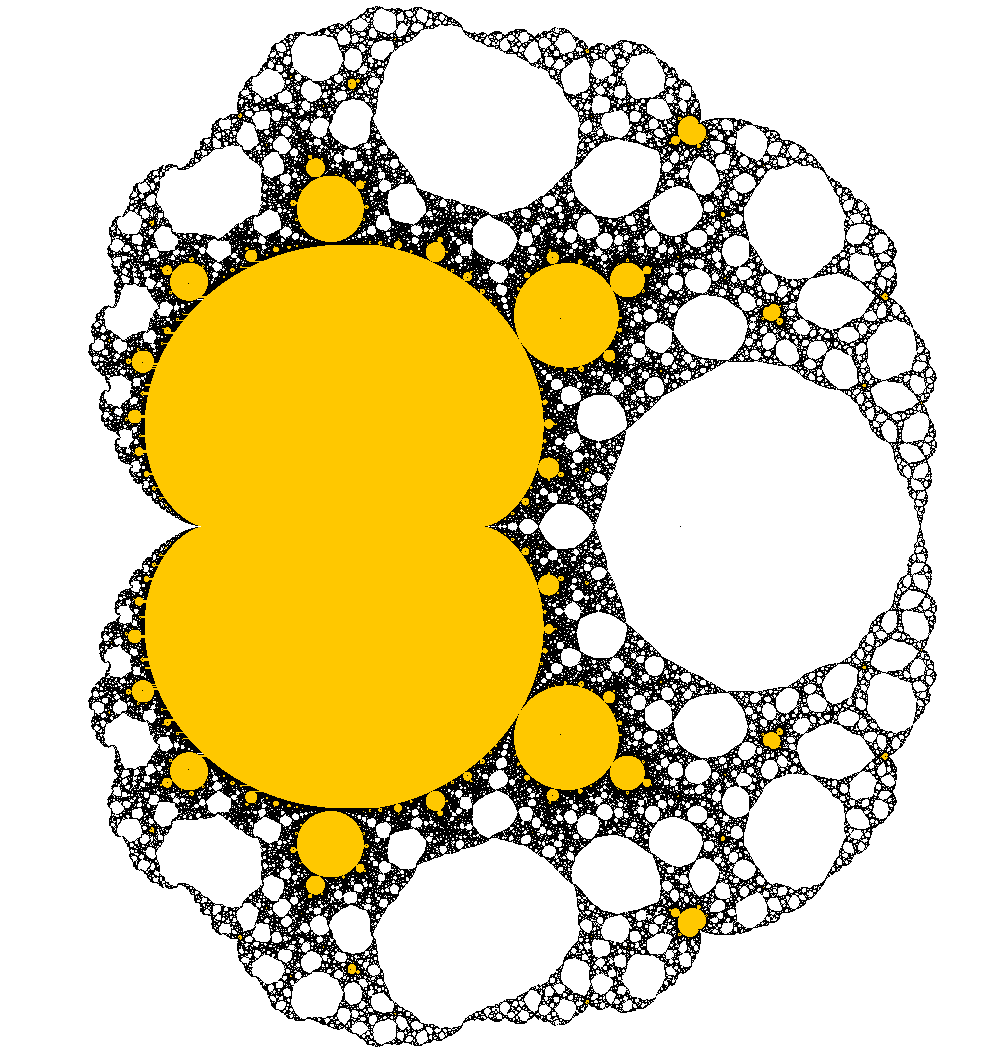}
  \includegraphics[width=70mm]{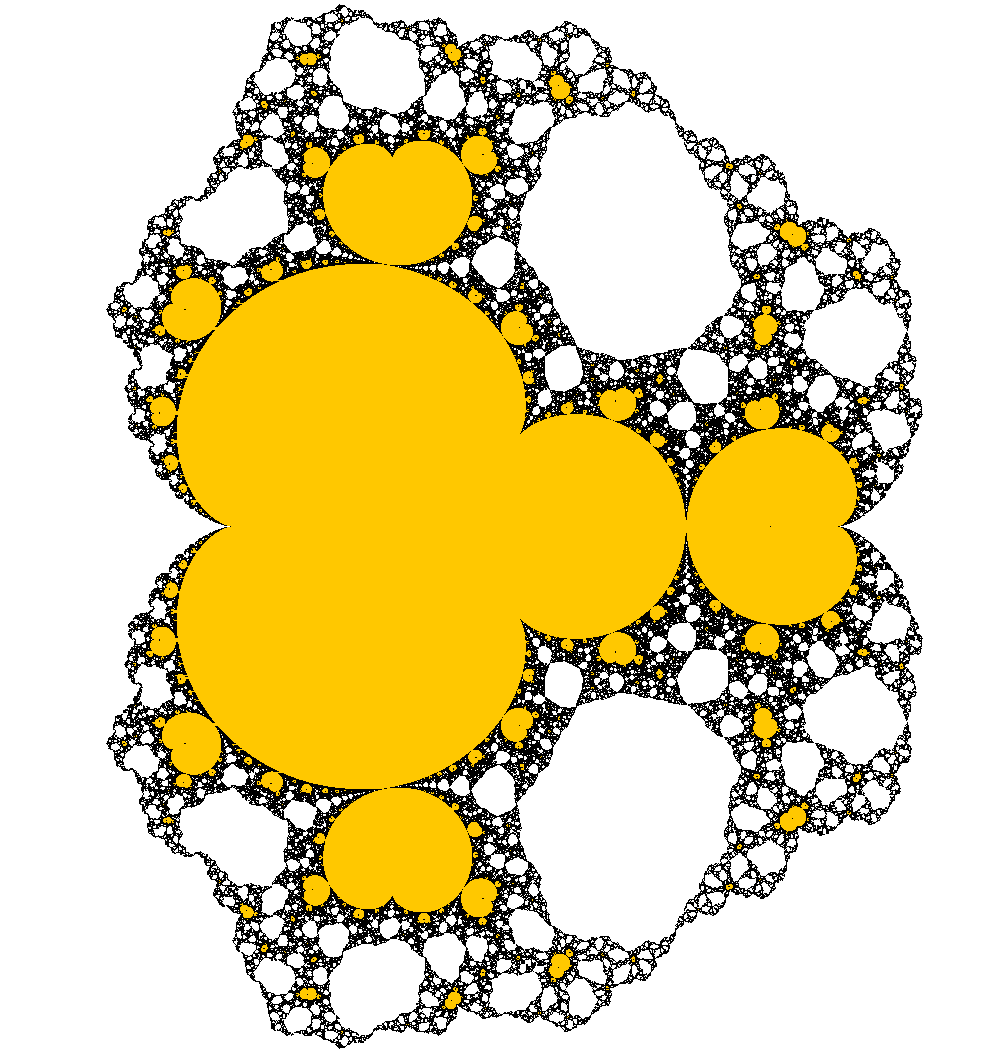}
  \caption{The non-escaping loci $\MM_2$ and $\MM_3$.}
  \label{Fig_parameter}
\end{figure}

The non-escaping locus $\MM_d$ can be identified as the complex plane cutting out infinitely many simply connected domains, which will be called `capture domains' later (see Figure \ref{Fig_parameter} and Proposition \ref{prop-decp}). There exist many small copies of the Mandelbrot set $M$ in $\MM_d$ which correspond to the renormalizable parameters.

For the connectivity of the non-escaping locus $\MM_d$, Wang et al. proved that $\MM_2$ is connected \cite[Theorem 1.1]{WQYQG}. We now generate this result to all $\MM_d$, where $d\geq 2$.

\begin{thm}\label{M-connected}
The non-escaping locus $\MM_d$ is connected for $d\geq 2$.
\end{thm}

The proof of the connectivity of $\MM_2$ in \cite{WQYQG} is based on constructing Riemann mapping from the capture domain to the unit disk $\D$, which is tediously long. Here, we give a proof of Theorem \ref{M-connected} by using the methods of Teichm\"{u}ller theory of the rational maps which was developed in \cite{McS}. The proof is largely simplified and there are several additional results. For example, we show that the Julia set of $U_{d\lambda}$ is a quasicircle if and only if $\lambda$ lies in the unbounded capture domain $\MH_0$ (Proposition \ref{prop-H-0-unbound}) and each bounded capture domain contains exactly one \emph{center} (Theorem \ref{thm-H_n-simp}).

If $\lambda$ is large enough, then the Julia set of $U_{d\lambda}$ is a quasicircle (see Proposition \ref{prop-H-0-unbound}). Hu and Lin observed that these circles become more and more `circular' as $\lambda$ tends to $\infty$ in the case of $d=2$ \cite{HL}. In \cite{Ga}, Gao proved the Hausdorff dimension of the Julia set of $U_{2n\lambda}$ tends to 1 for every $n\geq 2$, which gave an affirmative answer of Hu and Lin proposed in 1989. In this paper, we consider the asymptotic formula of the Hausdorff dimension of the Julia set $J_{d\lambda}$ of $U_{d\lambda}$ as the parameter $\lambda$ tends to $\infty$.

\begin{thm}\label{thm-Haus}
Let $d\geq 2$. For large $\lambda$ such that $J_{d\lambda}$ is a quasicircle, the Hausdorff dimension of $J_{d\lambda}$ is given by
\begin{equation}\label{asmy-fml}
\dim_H(J_{d\lambda})=1+\frac{1}{4\log{d}}|\lambda|^{-\frac{2}{d+1}}+\mathcal{O}(\lambda^{-\frac{3}{d+1}}).
\end{equation}
\end{thm}

The proof of the asymptotic formula \eqref{asmy-fml} is based on the calculation of an explicit iterated function system (see Lemma \ref{lema-fix}). As a useful tool, the iterated function system has been used to study the Hausdorff dimension of Julia sets in several papers previously. The first heart-stirring formula on the Hausdorff dimension of Julia sets, which was calculated by  an iterated function system, was due to Ruelle \cite{Ru}. He proved that for polynomials $P_c(z)=z^d+c$ with degree $d\geq 2$, if $c$ is small, then the Hausdorff dimension of the Julia set $J_c$ of $P_c$ is given by
\begin{equation}\label{Haus-H-J_c}
\dim_H(J_c)=1+\frac{|c|^2}{4\log d}+\mathcal{O}(c^3).
\end{equation}
Later, the Hausdorff dimension formula of $J_c$ was recalculated in \cite{WBKS} and \cite[p.\,119]{CDM}, where the formula \eqref{Haus-H-J_c} was expanded to the third order and fourth order in $c$, respectively.

We would like to mention that Theorem \ref{thm-Haus} is a generation of \cite{Os} in which the asymptotic formula of the Hausdorff dimension of $J_{2\lambda}$ was calculated. Recently, the iterated function system has been used to calculate the Hausdorff dimension of the boundary of the immediate basin of infinity of the McMullen maps \cite{YW}. Note that the iterated function system is just probably suitable for calculating the Hausdorff dimension of the quasicircles. Rather than iterated function system, Shishikura and Tan use renormalization theory to study the Hausdorff dimensions of the Julia sets and the bifurcation loci of parameter spaces. For example, see \cite{Sh} and \cite{Ta}.

This paper is organized as follows.
In \S\ref{locate-crit}, we analyze the location of the critical points of $U_{d\lambda}$ and show that the Julia set of $U_{d\lambda}$ is always connected and prove Theorem \ref{J-connected}.
In \S\ref{sec-no-sier}, we show that if the parameter lies on the real axis, then there exist two Fatou components of $U_{d\lambda}$ such that the intersection of the boundaries of them is non-empty and the Julia set of $U_{d\lambda}$ cannot be a Sierpi\'{n}ski carpet, which means Theorem \ref{thm-no-sier} holds.
In \S\ref{sec-decom}, we show that the parameter plane of $U_{d\lambda}$ can be decomposed into the non-escaping locus $\MM_d$ union infinitely many capture domains.
In \S\ref{sec-quasi-conj}, we give a complete classification of the quasiconformal conjugacy classes of $U_{d\lambda}$.
In \S\ref{sec-simp-con}, we show that each bounded capture domain is simply connected and the unique unbounded capture domain is homeomorphic to the punctured disk and prove Theorem \ref{M-connected}.
We will prove the asymptotic formula \eqref{asmy-fml} of Theorem \ref{thm-Haus} in \S\ref{pf-asym-fml} but leave the complicated calculations to the last section as an appendix.

\vskip0.2cm
\noindent\textit{Acknowledgements.} We want to express our deep thanks to the referee for his careful reading and pertinent comments which indeed improved this paper a lot.

\section{The location of critical points and the connected Julia sets}\label{locate-crit}

Firstly, we give a splitting principle for $U_{d\lambda}$. This principle is not exist if one considers $U_{mn\lambda}$ with $m\neq n$. This is the reason why we set $m=n$ in this paper. For every $\lambda\in\C^*$, it is straightforward to verify that $U_{d\lambda}=T_{d\lambda}\circ T_{d\lambda}$, where
\begin{equation}\label{T-d-lambda}
U_{d\lambda}(z)=\left(\frac{(z+\lambda-1)^d+(\lambda-1)(z-1)^d}{(z+\lambda-1)^d-(z-1)^d}\right)^d \text{~and~} T_{d\lambda}(z)=\left(\frac{z+\lambda-1}{z-1}\right)^d.
\end{equation}

A direct calculation shows that the set of all critical points of $T_{d\lambda}$ is $\{1,1-\lambda\}$, and both with multiplicity $d-1$. Note that
\begin{equation}\label{pre-1}
U_{d\lambda}^{-1}(\infty)=T_{d\lambda}^{-1}(1)=\bigcup_{k=0}^{d-1}\{\xi_k\} \text{~and~}
U_{d\lambda}^{-1}(0)=T_{d\lambda}^{-1}(1-\lambda)=\bigcup_{k=0}^{d-1}\{\omega_k\},
\end{equation}
where
\begin{equation}\label{pre-1-lambda}
\xi_k=\frac{e^{\frac{2k\pi i}{d}}+\lambda-1}{e^{\frac{2k\pi i}{d}}-1}  \text{~and~}
\omega_k=\frac{(1-\lambda)^{\frac{1}{d}}e^{\frac{2k\pi i}{d}}+\lambda-1}{(1-\lambda)^{\frac{1}{d}}e^{\frac{2k\pi i}{d}}-1}.
\end{equation}
It follows that $\xi_k$ and $\omega_k$ are critical points of $U_{d\lambda}$ with multiplicity $d-1$, where $0\leq k\leq d-1$. In particular, $\xi_0=\infty$. Therefore, the set of all critical points of $U_{d\lambda}$ is
\begin{equation}\label{Crit-U-lambda}
\Crit (U_{d\lambda})=\{1,1-\lambda,\infty\}\cup\bigcup_{k=1}^{d-1}\{\xi_k\}\cup\bigcup_{k=0}^{d-1}\{\omega_k\}.
\end{equation}

Since $T_{d\lambda}(1)=\infty$, $T_{d\lambda}(\infty)=1$ and $1,\infty$ are both critical points of $U_{d\lambda}$, it means that there exist two fixed immediate superattracting basins $\mathcal{A}_{d\lambda}(1)$ and $\mathcal{A}_{d\lambda}(\infty)$ of $U_{d\lambda}$ with centers $1$ and $\infty$ respectively. Under the iteration of $T_{d\lambda}$, we have the following forward orbits:
\begin{equation}\label{Orbit-Crit}
\xi_k\mapsto 1\mapsto\infty\mapsto 1\mapsto\infty\mapsto \cdots \text{~~and~~}
\omega_k\mapsto 1-\lambda\mapsto 0\mapsto (1-\lambda)^d\mapsto\cdots
\end{equation}
for every $0\leq k\leq d-1$. Since the dynamical behaviors are determined by the critical forward orbits essentially, we only need to focus on the \emph{free} critical orbit of $1-\lambda$ (or equivalently, the forward orbit of 0) under the iteration of $T_{d\lambda}$ or $U_{d\lambda}$. This is the reason why we define the non-escaping locus $\MM_d$ as in \eqref{Mandel}.

\begin{lema}\label{lema-covering}
Let $U$ and $V$ be two domains on $\EC$ and assume that $V$ is simply connected. If $f:U\rightarrow V$ is a branched covering with only one critical value in $V$ (counted without multiplicity), then $U$ is also simply connected.
\end{lema}

\begin{proof}
Let $v$ be the unique critical value lying in $V$. Consider the unramified covering $f:U\setminus f^{-1}(v)\rightarrow V\setminus\{v\}$. Since $V\setminus\{v\}$ is an annulus with Euler characteristic $0$, it follows that $U\setminus f^{-1}(v)$ is also an annulus by the Riemann-Hurwitz formula. This means that $U$ is a topological disk, which is simply connected as desired.
\end{proof}

In order to prove a rational map has connected Julia set, one often needs to exclude the existence of Herman ring. The following lemma was proved in \cite{Ya}.

\begin{lema}[{\cite[Corollary 3.2]{Ya}}]\label{no-Herman}
The renormalization transformation $U_{d\lambda}$ has no Herman ring.
\end{lema}

The proof of Lemma \ref{no-Herman} relies on the quasiconformal surgery and the arguments are divided into two cases: Herman ring with period $1$ and period at least two. However, the prove idea is different from \cite[Appendix A]{Mi2}.

\begin{thm}\label{J-connected-stated-again}
The Julia set of $T_{d\lambda}$ is always connected for every $d\geq 2$ and $\lambda\in\C^*$.
\end{thm}

\begin{proof}
The proof idea is more or less similar to the case of quadratic rational maps in \cite[Lemma 8.2]{Mi1}.
Note that the Julia set is connected if and only if each Fatou component is simply connected. By Sullivan's classification of the periodic Fatou components, every periodic Fatou component of $T_{d\lambda}$ is either a Siegel disk, a Herman ring, or an immediate basin for some attracting or parabolic point. By Lemma \ref{no-Herman}, it is known $T_{d\lambda}$ has no Herman ring.

By \cite[Lemma 8.1]{Mi1}, we know that if all the critical values of a rational map are contained in a single component of the Fatou set, then the Julia set is totally connected. However, the Julia set $J_{d\lambda}$ cannot be totally disconnected since $T_{d\lambda}$ has a superattracting periodic orbit of period 2. Therefore, the critical points $1$ and $1-\lambda$ lie in different Fatou components and each Fatou component of $T_{d\lambda}$ contains at most one critical value ($\infty$ or $0$ by \eqref{Orbit-Crit}).

Now we prove each Fatou component of $T_{d\lambda}$ is simply connected. Firstly, we assume that every periodic Fatou component of $T_{d\lambda}$ is simply connected. Note that the periodic orbit $1\leftrightarrow \infty$ is superattracting. There leaves only one critical point $1-\lambda$ needing to consider. According to Lemma \ref{lema-covering}, the preimage of a simply connected region under a branched covering with only one critical value is again simply connected. This means every Fatou component of $T_{d\lambda}$ is simply connected by induction.

Then suppose that there exists a periodic Fatou component $U$ of $T_{d\lambda}$ which is not simply connected and the period is $p\geq 1$. This means that $U$ is an attracting basin or a parabolic basin since $T_{d\lambda}$ has no Herman ring. Let $z_0$ be the attracting periodic point in $U$ or parabolic periodic point on $\partial U$. We use $V$ to denote a simply connected neighborhood or a simply connected petal of $z_0$ such that $T_{d\lambda}^{\circ p}(V)\subset V$ according to $U$ is attracting or parabolic. Let $V_k$ be the component of $T_{d\lambda}^{-kp}(V)$ containing $V$. Then $U=\bigcup_{k\geq 0}V_k$ and $V_{k+1}\mapsto T_{d\lambda}(V_{k+1})\mapsto\cdots\mapsto T_{d\lambda}^{\circ p-1}(V_{k+1})\mapsto V_k$ is a successive branched covering under $T_{d\lambda}$ with at most one critical value in each codomain since each Fatou component of $T_{d\lambda}$ contains at most one critical value. Suppose $V_{k_0}$ is simply connected (at least $k_0=0$ is satisfied). By Lemma \ref{lema-covering}, we know that $T_{d\lambda}^{\circ p-1}(V_{k_0+1}),\cdots, T_{d\lambda}(V_{k_0+1}),V_{k_0+1}$ are all simply connected since $V_{k_0}$ is also. Inductively, it follows that each $V_k$ is simply connected and hence $U$ is also simply connected. This contradicts the assumption that $U$ is not simply connected.

Therefore, in any case, the Julia set of $T_{d\lambda}$ is always connected. This ends the proofs of Theorems \ref{J-connected-stated-again} and \ref{J-connected}.
\end{proof}

\section{The Julia set cannot be a Sierpi\'{n}sk carpet}\label{sec-no-sier}

In this section, we will prove that if the parameter $\lambda$ lies on the real axis, then the Julia set of $U_{d\lambda}$ can never be a Sierpi\'{n}sk carpet by showing there always exist two Fatou components of $U_{d\lambda}$ whose boundaries are intersecting to each other.

\begin{lema}\label{lema-non-empty}
For every $d\geq 2$ and $\lambda\in \R$, there exist two Fatou components $V_1,V_2$ of $U_{d\lambda}$ such that $\overline{V}_1\cap\overline{V}_2\neq\emptyset$.
\end{lema}

\begin{proof}
If $\lambda=0$, then $U_{d\lambda}$ degenerates to a parabolic polynomial $U_{d0}(z)=(\frac{z+d-1}{d})^d$ whose Julia set $J_{d0}$ is a Jordan curve. Let $V_1=\mathcal{A}_{d\lambda}(1)$ and $V_2=\mathcal{A}_{d\lambda}(\infty)$ be the immediate superattracting basins of $1$ and $\infty$ respectively. We have $\overline{V}_1\cap\overline{V}_2=J_{d0}\neq\emptyset$.

In the following, we assume that $\lambda\in\R\setminus\{0\}$. The dynamics of $U_{d\lambda}$ will be restricted on the real axis and the arguments will be divided into several cases. Let $x\in \R$, by a direct calculation, we have
\begin{equation}\label{U-d-lambda}
U_{d\lambda}'(x)=\frac{d^2\lambda^2(x-1)^{d-1}(x+\lambda-1)^{d-1}((x+\lambda-1)^d+(\lambda-1)(x-1)^d)^{d-1}}
{((x+\lambda-1)^d-(x-1)^d)^{d+1}}.
\end{equation}

(1) Let $\lambda>0$. If $x\geq 1$, we have $x-1\geq 0$, $x+\lambda-1>0$, $(x+\lambda-1)^d+(\lambda-1)(x-1)^d>0$ and $(x+\lambda-1)^d-(x-1)^d>0$. This means that $U_{d\lambda}'(x)\geq 0$ and $U_{d\lambda}$ is increasing on $[1,+\infty)$. Moreover, $U_{d\lambda}'(x)= 0$ if and only if $x=1$. We claim that there exists at least one fixed point of $U_{d\lambda}$ lying in $(1,+\infty)$. Otherwise, we then have $1<U_{d\lambda}(x)<x$ for every $x>1$ since $U_{d\lambda}(1)=1$ and $U_{d\lambda}'(1)=0$. This means that the interval $(1,+\infty)$ is contained in the attracting basin of $1$, which is a contradiction since $\infty$ is a superattracting fixed point of $U_{d\lambda}$.

Let $1=x_0<x_1<\cdots<x_n<+\infty$ be the collection of all the fixed points of $U_{d\lambda}$ lying in $[1,+\infty)$, where $n\geq 1$. It is easy to see $U_{d\lambda}(x)>x$ if $x>x_n$. In particular, we have $(x_n,+\infty)\subset\mathcal{A}_{d\lambda}(\infty)$. Note that $U_{d\lambda}'(x_n)\geq 1$. If $U_{d\lambda}'(x_n)=1$, then $x_n$ is a parabolic fixed point of $U_{d\lambda}$ and $\mathcal{A}_{d\lambda}(x_n)$ contains a small interval on the left of $x_n$, where $\mathcal{A}_{d\lambda}(x_n)$ is the immediate parabolic basin of $x_n$. Let $V_1=\mathcal{A}_{d\lambda}(x_n)$ and $V_2=\mathcal{A}_{d\lambda}(\infty)$. We have $x_n\in\overline{V}_1\cap\overline{V}_2$. If $U_{d\lambda}'(x_n)>1$, then $x_n$ is a repelling fixed point of $U_{d\lambda}$ and $x_{n-1}$ is an (or parabolic) attracting fixed point of $U_{d\lambda}$. Moreover, $[x_{n-1},x_n)\subset\mathcal{A}_{d\lambda}(x_{n-1})$, where $\mathcal{A}_{d\lambda}(x_{n-1})$ is the immediate attracting (or parabolic) basin of $x_{n-1}$. Let $V_1=\mathcal{A}_{d\lambda}(x_{n-1})$ and $V_2=\mathcal{A}_{d\lambda}(\infty)$. We have $x_n\in\overline{V}_1\cap\overline{V}_2$.

(2) Let $\lambda<0$. If $0\leq x\leq 1$, then $x-1\leq 0$ and $x+\lambda-1<0$. If $d\geq 2$ is even, then $(x+\lambda-1)^d+(\lambda-1)(x-1)^d>0$, $(x+\lambda-1)^d-(x-1)^d>0$ and $U_{d\lambda}'(x)\geq 0$. If $d\geq 2$ is odd, then $U_{d\lambda}'(x)\geq 0$. This means that $U_{d\lambda}$ is increasing on $[0,1]$ for every $d\geq 2$. Moreover, $U_{d\lambda}'(x)= 0$ if and only if $x=1$. By a straightforward calculation, we have $0<U_{d\lambda}(0)<1$. Now we divide the arguments into two cases.

If there exists no fixed point of $U_{d\lambda}$ in $(0,1)$, then we have $0<x<U_{d\lambda}(x)<1$ for every $0<x<1$. This means that $0$ lies in the immediate attracting basin of $1$. By Lemma \ref{equiv-condi}(5), we know that $J_{d\lambda}$ is a quasicircle. In particular, $\overline{\mathcal{A}_{d\lambda}(1)}\cap\overline{\mathcal{A}_{d\lambda}(\infty)}=J_{d\lambda}\neq\emptyset$. If there exists at least one fixed point of $U_{d\lambda}$ in $(0,1)$, we denote all of them by $0<x_1<\cdots<x_n<1$, where $n\geq 1$. By a completely similar argument as the case $\lambda>0$, one can show that the fixed point $x_n$ is contained in the boundaries of two different Fatou components. Therefore, the proof is complete.
\end{proof}

\begin{thm}\label{thm-no-sier-resta}
For every $d\geq 2$ and $\lambda\in \R$, the Julia set $J_{d\lambda}$ is not a Sierpi\'{n}sk carpet.
\end{thm}

\begin{proof}
Note that if $J_{d\lambda}$ is a Sierpi\'{n}ski carpet, then the closure of any two Fatou components of $U_{d\lambda}$ cannot be intersecting to each other. But this contradicts Lemma \ref{lema-non-empty}. The proofs of Theorems \ref{thm-no-sier-resta} and \ref{thm-no-sier} are finished.
\end{proof}

\begin{rmk}
By computer experiments, it is shown that $\overline{\mathcal{A}_{d\lambda}(1)}\cap\overline{\mathcal{A}_{d\lambda}(\infty)}=\{z_0\}$ for $\lambda\in\C$, where $z_0$ is a repelling fixed point of $U_{d\lambda}$. Therefore, the Julia set $J_{d\lambda}$ can never be a Sierpi\'{n}sk carpet for any $\lambda\in\C$ (see Figures \ref{Fig_no-carpet} and \ref{Fig_Julia-Siegel}).
\end{rmk}

\section{Decomposition of the parameter space}\label{sec-decom}

In this section, we divide the parameter space of $T_{d\lambda}$ into the non-escaping locus $\MM_d$ union countably many capture domains. Recall that $\mathcal{A}_{d\lambda}(1)$ and $\mathcal{A}_{d\lambda}(\infty)$ are the immediate superattracting basins of $1$ and $\infty$ respectively.

\begin{lema}\label{equiv-condi}
For each $\lambda\in\C^*$, the following conditions are equivalent:

$(1)$ The Julia set $J_{d\lambda}$ of $T_{d\lambda}$ is a quasicircle;
$(2)$ $\xi_k\in\mathcal{A}_{d\lambda}(\infty)$ for all $0\leq k\leq d-1$;
$(3)$ $\omega_k\in\mathcal{A}_{d\lambda}(1)$ for all $0\leq k\leq d-1$;
$(4)$ $1-\lambda\in\mathcal{A}_{d\lambda}(\infty)$;
$(5)$ $0\in\mathcal{A}_{d\lambda}(1)$.

In particular, $\omega_k\in\mathcal{A}_{d\lambda}(1)$ if and only if $\omega_l\in\mathcal{A}_{d\lambda}(1)$, where $0\leq k,l\leq d-1$.
\end{lema}

\begin{proof}
We first prove $(1)\Rightarrow (2)(3)(4)(5)$. If $J_{d\lambda}$ is a quasicircle, the Fatou set of $T_{d\lambda}$ consists of two simply connected Fatou components $A_{d\lambda}(1)$ and $A_{d\lambda}(\infty)$ whose common boundary is $J_{d\lambda}$. Since $T_{d\lambda}$ permutes $1$ and $\infty$, by \eqref{Orbit-Crit}, it follows that (2) holds and $\{\omega_1,\cdots,\omega_d\}$ lies in a single Fatou component. Applying the Riemann-Hurwitz formula to $U_{d\lambda}:\A_{d\lambda}(\infty)\rightarrow\A_{d\lambda}(\infty)$, it follows that $\{\omega_1,\cdots,\omega_d,0\}\subset\A_{d\lambda}(1)$ and $1-\lambda\in \A_{d\lambda}(\infty)$. Therefore, (3)(4)(5) hold.

By \eqref{Orbit-Crit}, we have $(3)\Rightarrow (4)\Rightarrow (5)$. Now we prove $(5)\Rightarrow (1)$. Suppose that $0\in\mathcal{A}_{d\lambda}(1)$. By \eqref{pre-1}, we have $U_{d\lambda}^{-1}(0)=\bigcup_{k=0}^{d-1}\{\omega_k\}$. Since $U_{d\lambda}(\A_{d\lambda}(1))=\A_{d\lambda}(1)$, there exists some $k_0$ such that $\omega_{k_0}\in\A_{d\lambda}(1)$ and hence $1-\lambda\in\A_{d\lambda}(\infty)$. Note that $T_{d\lambda}:\A_{d\lambda}(1)\rightarrow \A_{d\lambda}(\infty)$ is $d$ to $1$. We claim that $\omega_k\in\A_{d\lambda}(1)$ for every $0\leq k\leq d-1$. In fact, if not, then $1-\lambda$ has at least $d+1$ preimages under $T_{d\lambda}$ (counted with multiplicity, $d$ in $\A_{d\lambda}(1)$ and at least one elsewhere), which is impossible. The same argument also shows that $\omega_k\in\mathcal{A}_{d\lambda}(1)$ if and only if $\omega_l\in\mathcal{A}_{d\lambda}(1)$, where $0\leq k,l\leq d-1$. Then, $\A_{d\lambda}(1)$ contains critical points $\{\omega_1,\cdots,\omega_d,1\}$ of $U_{d\lambda}$. This means that $\A_{d\lambda}(1)$ is completely invariant under $U_{d\lambda}$.

Since $1-\lambda\in\A_{d\lambda}(\infty)$, it means that $T_{d\lambda}:\A_{d\lambda}(\infty)\rightarrow \A_{d\lambda}(1)$ is $d$ to $1$. Therefore, $\xi_k\in\A_{d\lambda}(\infty)$ for every $1\leq k\leq d-1$ since $\xi_0=\infty\in\A_{d\lambda}(\infty)$ and $T_{d\lambda}(\xi_k)=1$. Moreover, $\A_{d\lambda}(\infty)$ contains critical points $\{\xi_1,\cdots,\xi_d,1-\lambda\}$ of $U_{d\lambda}$. This means that $\A_{d\lambda}(\infty)$ is also completely invariant under $U_{d\lambda}$. Therefore, $J_{d\lambda}$ is a quasicircle since $T_{d\lambda}$ is hyperbolic and $T_{d\lambda}$ has exactly two Fatou components. This ends the proof of $(5)\Rightarrow (1)$.

To finish, we prove $(2)\Rightarrow (4)$. If $\xi_k\in\mathcal{A}_{d\lambda}(\infty)$ for all $0\leq k\leq d-1$, then $T_{d\lambda}:\mathcal{A}_{d\lambda}(\infty)\rightarrow\mathcal{A}_{d\lambda}(1)$ is $d$ to $1$. This means that $1-\lambda\in\A_{d\lambda}$ by Riemann-Hurwitz formula. The proof is complete.
\end{proof}

\begin{lema}\label{lema-0-1-lambda}
For every $\lambda\in\C^*$, we have $0\not\in\A_{d\lambda}(\infty)$ and $1-\lambda\not\in\A_{d\lambda}(1)$.
\end{lema}

\begin{proof}
If $0\in\A_{d\lambda}(\infty)$, then $1-\lambda\in\A_{d\lambda}(1)$ by \eqref{Orbit-Crit}. Note that $1$ lies also in $\A_{d\lambda}(1)$. This means that $T_{d\lambda}$ has $2d-1$ preimages in $\A_{d\lambda}(1)$ for each point in $\A_{d\lambda}(\infty)$ by Riemann-Hurwitz formula, which is a contradiction. Moreover, $0\not\in\A_{d\lambda}(\infty)$ means $1-\lambda\not\in\A_{d\lambda}(1)$ by \eqref{Orbit-Crit}.
\end{proof}

Since $1$ and $\infty$ are always periodic with period 2 under $T_{d\lambda}$, the \emph{non-escaping locus} $\MM_d$ associated to $T_{d\lambda}$ can be defined as
\begin{equation}\label{Mandel-new}
\MM_d=\{\lambda\in\C^*:T_{d\lambda}^{\circ 2n}(0)\not\rightarrow 1 \text{~and~} T_{d\lambda}^{\circ 2n+1}(0)\not\rightarrow 1 \text{~as~}n\rightarrow\infty\}\cup\{0\}.
\end{equation}

\begin{defi}\label{def-H-n}
Define $\mathcal{H}_0:=\{\lambda\in\C^*:0\in\A_{d\lambda}(1)\}$. For every $n\geq 1$, define
\begin{equation}
\mathcal{H}_n:=\{\lambda\in\C^*:T_{d\lambda}^{\circ n}(0)\in\A_{d\lambda}(1)\text{~and~}T_{d\lambda}^{\circ n-1}(0)\not\in\A_{d\lambda}(\infty)\}.
\end{equation}
Each component of $\mathcal{H}_n$ is called a \emph{capture domain} of \emph{depth} $n$, where $n\geq 0$.
\end{defi}

\begin{prop}\label{prop-decp}
The parameter space of $T_{d\lambda}$ has the following decomposition:
\begin{equation}
\C=\MM_d\sqcup (\bigsqcup_{n\geq 0}\mathcal{H}_n).
\end{equation}
\end{prop}

\begin{proof}
By definitions of the non-escaping locus and $\MH_n$, we have $\MM_d\cap (\bigcup_{n\geq 0}\mathcal{H}_n)=\emptyset$. We need to show that two capture domains with different depths are disjoint and each $\lambda\in\C\setminus\MM$ belongs to $\mathcal{H}_n$ for some $n\geq 0$. First, suppose that $\lambda\in \MH_m\cap\MH_n$ for $m\neq n$. Without loss of generality, assume that $m>n\geq 0$. By Definition \ref{def-H-n}, we have $T_{d\lambda}^{\circ n}(0)\in\A_{d\lambda}(1)$ and $T_{d\lambda}^{\circ m-1}(0)\not\in\A_{d\lambda}(\infty)$. This means that $T_{d\lambda}^{\circ m-1}(0)\in\A_{d\lambda}(1)$ and hence $T_{d\lambda}^{\circ m}(0)\in\A_{d\lambda}(\infty)$, which contradicts $T_{d\lambda}^{\circ m}(0)\in\A_{d\lambda}(1)$. Therefore $\MH_m\cap\MH_n=\emptyset$ for $m\neq n$.

By \eqref{Mandel-new}, if $\lambda\not\in\MM_d$, there exists a minimal $k\geq 0$ such that $T_{d\lambda}^{\circ k}(0)\in\A_{d\lambda}(1)$. If $k=0$, then $\lambda\in\MH_0$. If $k=1$, then $T_{d\lambda}(0)\in\A_{d\lambda}(1)$. Lemma \ref{lema-0-1-lambda} asserts that $0\not\in\A_{d\lambda}(\infty)$. Therefore, $\lambda\in\MH_1$ in this case. If $k\geq 2$, we claim that $T_{d\lambda}^{\circ k-1}(0)\not\in\A_{d\lambda}(\infty)$. In fact, if not, we have $T_{d\lambda}^{\circ k-2}(0)\in\A_{d\lambda}(1)$. This contradicts the choice of the integer $k$. So we have $\lambda\in\MH_k$ in this case. The proof is complete.
\end{proof}

See Figure \ref{Fig_parameter} for the non-escaping loci $\MM_2$ and $\MM_3$. There some capture domains are also clearly visible (blank regions).

\section{Quasiconformal conjugacy classes}\label{sec-quasi-conj}

Let $\mathcal{R}_d$ be the collection of all $T_{d\lambda}$, where $\lambda\in\C^*$. In this section, we give a complete characterization of the quasiconformal conjugacy classes in $\mathcal{R}_d$.

\begin{defi}
Let $\Lambda$ be a complex manifold. A \emph{holomorphic family} of rational maps \emph{parameterized} by $\Lambda$ is a holomorphic map $f_\lambda:\Lambda\times\EC\rightarrow\EC$ such that $f_\lambda(z)$ is a rational map for fixed $\lambda\in\Lambda$ and depends holomorphically on $\lambda\in\Lambda$ for fixed $z\in\EC$.
\end{defi}

The parameter $\lambda\in\Lambda$ is called a $J$-\emph{stable} parameter of a holomorphic family of rational maps $f_\lambda$ if the total number of attracting cycles of $f_\lambda$ is constant in a neighborhood of $\lambda$.

\begin{thm}\label{thm-J-stable}
The boundary $\partial\MM_d$ is the set of parameters such that $T_{d\lambda}$ are not $J$-stable in $\mathcal{R}_d$.
\end{thm}

\begin{proof}
By \cite[Theorem 4.2]{Mc}, $T_{d\lambda_0}$ is $J$-stable if and only if both critical sequences $\{T_{d\lambda}^{\circ k}(1-\lambda)\}_{k\geq 0}$ and $\{T_{d\lambda}^{\circ k}(1)\}_{k\geq 0}$ are normal for $\lambda$ in a neighborhood of $\lambda_0$. Since $\{T_{d\lambda}^{\circ k}(1)\}_{n\geq 0}$ lies in a finite orbit $1\leftrightarrow \infty$, we only need to consider the orbit of $1-\lambda$. If $\lambda_0\in\MH_n$ for some $n\geq 0$, the orbit of $1-\lambda_0$ will be attracted by the cycle $1\leftrightarrow \infty$. For $\lambda$ close to $\lambda_0$, the orbit of $1-\lambda$ still converges to the cycle $1\leftrightarrow \infty$. By Montel's theorem, $\{T_{d\lambda}^{\circ k}(1-\lambda)\}_{k\geq 0}$ is normal at $\lambda_0$. Similarly, $\{T_{d\lambda}^{\circ k}(1-\lambda)\}_{k\geq 0}$ is normal at each point in the interior of $\MM_d$ since $\{T_{d\lambda}^{\circ k}(1-\lambda)\}_{k\geq 0}$ is disjoint with the attracting basin of $1\leftrightarrow\infty$. This means that $T_{d\lambda}$ is $J$-stable in $\C\setminus\partial\MM_d$.

On the other hand, if $\lambda_0\in\partial\MM_d$, then $\{T_{d\lambda_0}^{\circ k}(1-\lambda)\}_{k\geq 0}$ omits the attracting basin of $1\leftrightarrow\infty$. However, there are arbitrary small perturbation of $\lambda_0$ such that $\{T_{d\lambda}^{\circ k}(1-\lambda)\}_{k\geq 0}$ converges to the cycle $1\leftrightarrow\infty$. This means that $T_{d\lambda}$ is not $J$-stable on $\partial\MM_d$.
\end{proof}

\begin{cor}\label{cor-hyper}
Let $W$ be a component in the interior of $\MM_d$. If there exists $\lambda_0\in W$ such that $1-\lambda_0$ converges to an attracting cycle, then every $\lambda\in W$ also has this property.
\end{cor}

\begin{proof}
By Theorem \ref{thm-J-stable}, every $T_{d\lambda}\in W$ is $J$-stable. This means that there exists a small neighborhood of $\lambda$ such the number of attracting cycles is constant. Since $1-\lambda_0$ converges to an attracting cycle, this means that the constant is 2. The corollary follows.
\end{proof}

In the case of Corollary \ref{cor-hyper}, $W$ is called a \emph{hyperbolic} component. Otherwise, $W$ is called a \emph{queer} component. It was generally believed that queer components do not exist. But if they do, then every $T_{d\lambda}$ admits an invariant line field on its Julia set and the Julia set has positive Lebesgue area. See Figures \ref{Fig_no-carpet} and \ref{Fig_Julia-Siegel} for various Julia sets of $J_{d\lambda}$.

\begin{figure}[!htpb]
  \setlength{\unitlength}{1mm}
  \centering
  \includegraphics[width=70mm]{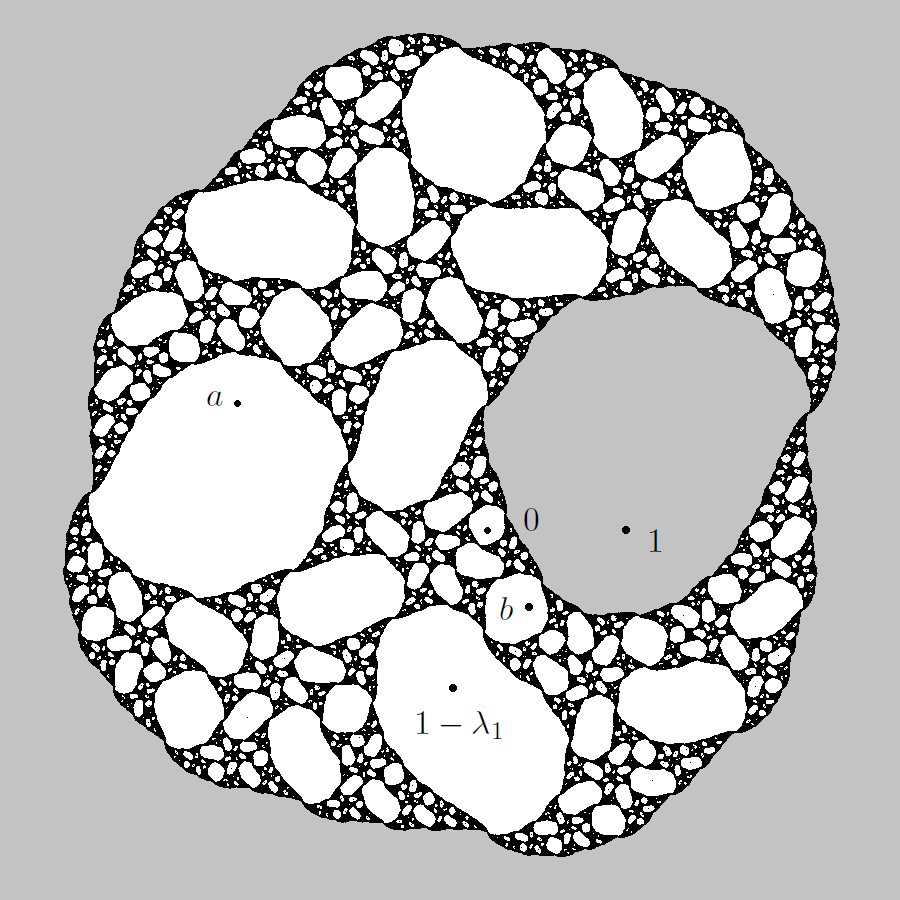}
  \includegraphics[width=70mm]{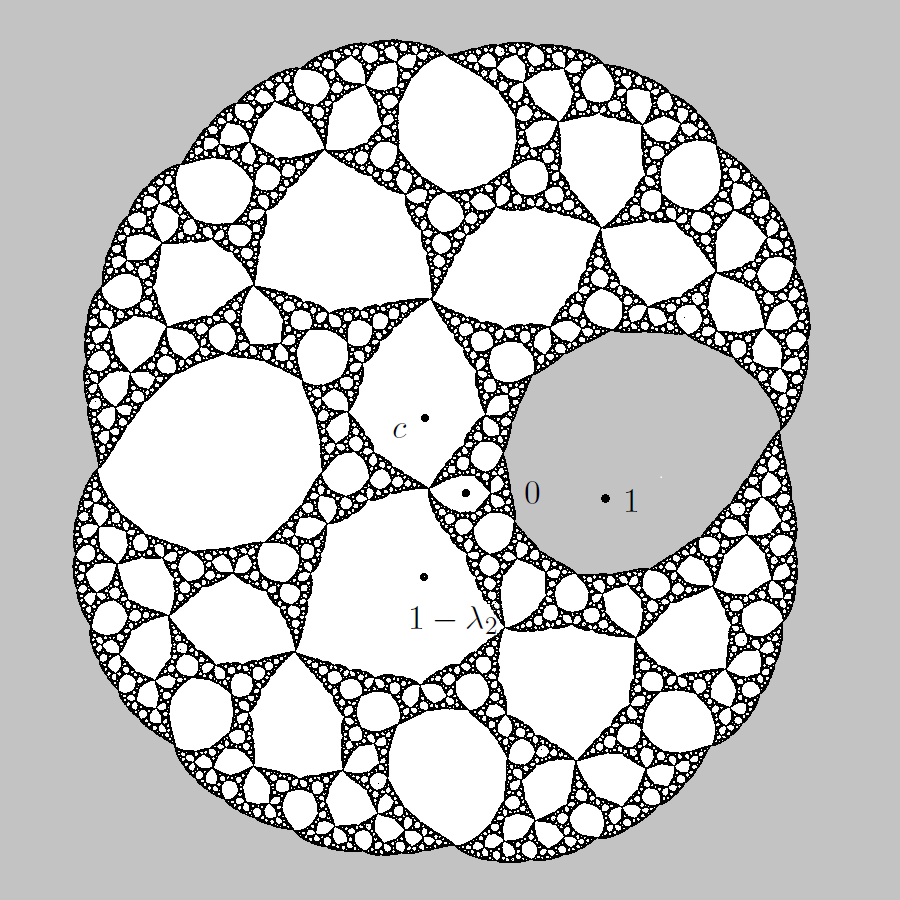}
  \caption{Julia sets of $T_{2\lambda}$ with $\lambda_1\approx 1.319448+1.633170i$ and $\lambda_2\approx 1.5+0.866025i$. The critical orbit $1\leftrightarrow\infty$ captures the critical orbit $1-\lambda_1\mapsto 0\mapsto a\mapsto b\mapsto 1$ and disjoint with the critical orbit $1-\lambda_2\mapsto 0\mapsto c\mapsto 1-\lambda_2$.}
  \label{Fig_no-carpet}
\end{figure}

\begin{figure}[!htpb]
  \setlength{\unitlength}{1mm}
  \centering
  \includegraphics[width=70mm]{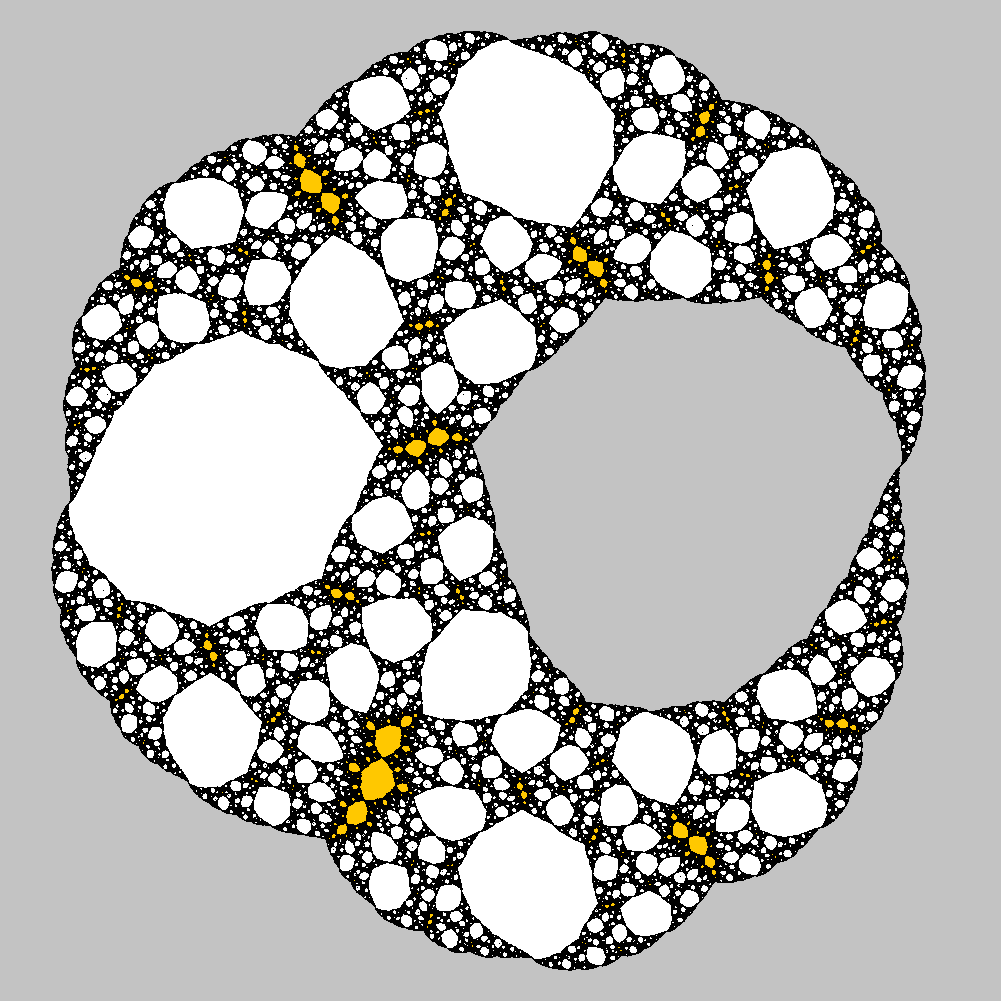}
  \includegraphics[width=70mm]{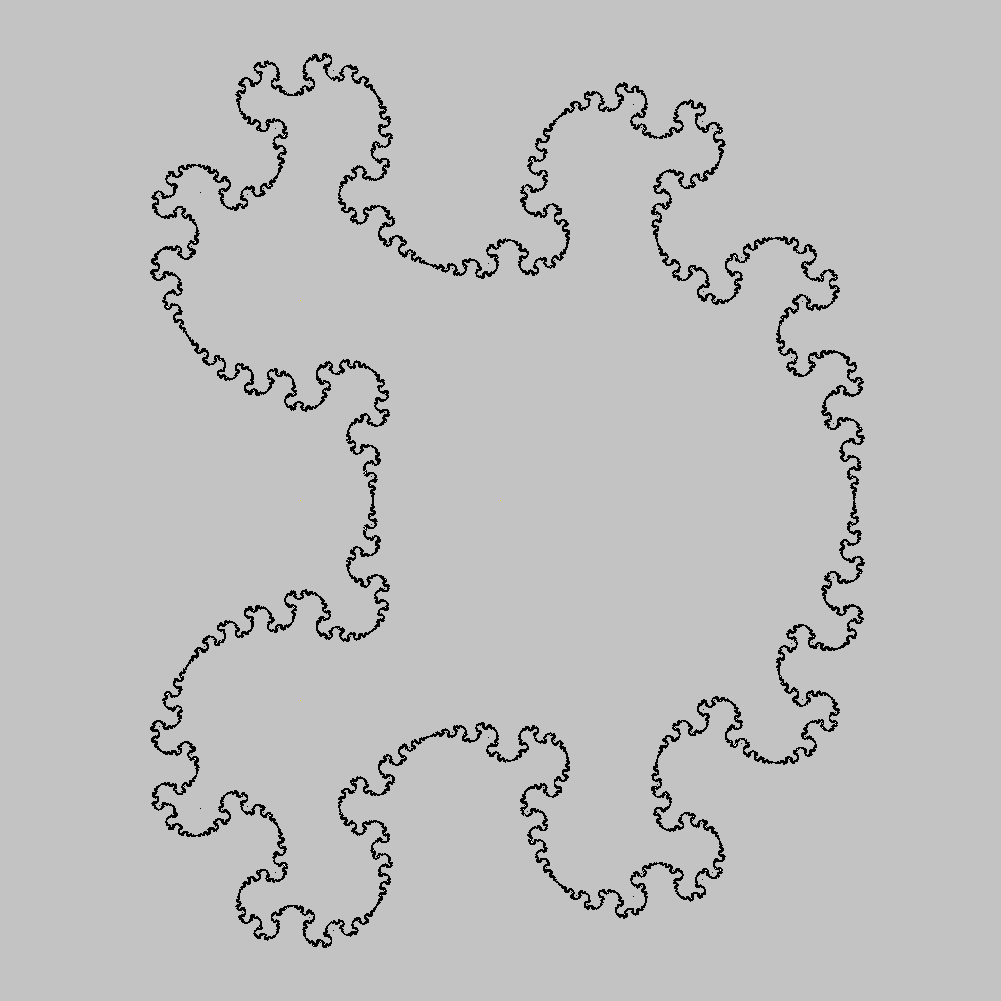}
  \caption{Julia sets of $T_{2\lambda}$ with $\lambda_3\approx 2.046736+1.589069i$ and $\lambda_4=4.0$. $T_{2\lambda_3}$ has a Siegel disk with periodic $4$ and $J_{2\lambda_4}$ is a quasicircle.}
  \label{Fig_Julia-Siegel}
\end{figure}

Now we state a theorem of parameterization of quasiconformal conjugacy classes.

\begin{thm}\label{thm-quasi}
Let $T_{d\lambda_0},T_{d\lambda_1}\in\mathcal{R}_d$ be two different maps and let $\varphi:\EC\rightarrow\EC$ be a $K$ -quasiconformal homeomorphism which conjugates $T_{d\lambda_0}$ to $T_{d\lambda_1}$ such that $\varphi(\lambda_0)=\lambda_1$. Then there exists a holomorphic map $t\mapsto \lambda_t$ from an open disk $\D(0,r)$ $(r>1)$ into $\C^*$ which maps $0$ to $\lambda_0$ and $1$ to $\lambda_1$, such that for every $t\in\D(0,r)$, $T_{d\lambda_0}$ is conjugate to $T_{d\lambda_t}$ by a $K_t$ -quasiconformal mapping $\varphi_t:\EC\rightarrow\EC$. Moreover, $K_t\rightarrow 1$ as $t\rightarrow 0$.
\end{thm}

The idea of the proof of Theorem \ref{thm-quasi} is standard in holomorphic dynamics. One can refer \cite[Theorem 5.1]{Za} for a proof in the similar situation. As an immediate corollary, we have

\begin{cor}\label{cor-quasi-rid}
Quasiconformal conjugacy classes in $\mathcal{R}_d$ are either single points or open and connected. In particular, the conjugacy classes on $\partial\MM_d$ are single points.
\end{cor}

A holomorphic family of rational maps $f_\lambda:\Lambda\times\EC\rightarrow\EC$ is \emph{quasiconformally constant} if $f_{\lambda_1}$ and $f_{\lambda_2}$ are quasiconformally conjugate for any $\lambda_1$ and $\lambda_2$ in the same component of $\Lambda$.
We call the family $f_\lambda$ has \emph{constant critical orbit relations} if any coincidence $f_\lambda^{\circ n}(c_1)=f_\lambda^{\circ m}(c_2)$  between the forward orbits of two critical points $c_1$ and $c_2$ of $f_\lambda$ persists under perturbation of $\lambda$.
The following theorem was proved in \cite[Theorem 2.7]{McS}.

\begin{thm}[{\cite{McS}}]\label{thm-const-crit}
A holomorphic family $f_\lambda$ of rational maps with constant critical orbit relations is quasiconformally constant.
\end{thm}

\begin{prop}\label{prop-H-0-unbound}
The Julia set $J_{d\lambda}$ of $T_{d\lambda}$ is a quasicircle if and only if $\lambda\in\MH_0$. Moreover, $\MH_0$ is unbounded and connected.
\end{prop}

\begin{proof}
By the definition of $\MH_0$ and Lemma \ref{equiv-condi}, it follows that if $\lambda\in\MH_0$, then $J_{d\lambda}$ is a quasicircle. Conversely, if $J_{d\lambda}$ is a quasicircle, then $1-\lambda\in\A_{d\lambda}(\infty)$. This means that $T_{d\lambda}$ and $T_{d\lambda_0}$ have the same critical orbit relations, where $\lambda_0\in \MH_0$. By Theorem \ref{thm-const-crit}, $T_{d\lambda}$ and $T_{d\lambda_0}$ are quasiconformally conjugate to each other. By Corollary \ref{cor-quasi-rid}, it follows that $\lambda\in\MH_0$ and $\MH_0$ is connected.

To finish, we only need to show that $\MH_0$ is unbounded. Let $\alpha=\lambda^{-\frac{1}{d+1}}$ and $\varphi_\alpha(z)=\alpha^d(z-1)$ be a linear transformation. By a straightforward calculation, we have
\begin{equation*}
f_\alpha(z):= \varphi_\alpha\circ T_{d\lambda}\circ\varphi_\alpha^{-1}=\sum_{i=0}^{d-1}\frac{C_d^i\,\alpha^i}{z^{d-i}}
=\frac{1}{z^d}+\frac{C_d^1\alpha}{z^{d-1}}+\cdots+\frac{C_d^1\alpha^{d-1}}{z}.
\end{equation*}
If $\alpha\neq 0$ is small enough, then the Julia set of $f_\alpha$ is a quasicircle since the Julia set of $z\mapsto 1/z^d$ is the unit circle. This means that $J_{d\lambda}$ is a quasicircle if $\lambda$ is large enough.
\end{proof}

By definition, the parameter $\lambda\in\bigcup_{n\geq 0}\MH_n$ if and only if the critical orbit $1-\lambda\mapsto 0\mapsto (1-\lambda)^d\mapsto\cdots$ tends to the attracting periodic cycle $1\mapsto\infty\mapsto 1$. A point $\lambda$ is called a \emph{center} of a hyperbolic component $W\subset\MM_d$ if the critical point $1-\lambda$ is periodic. On the other hand, $\lambda$ is called a \emph{center} of a capture domain of $\bigcup_{n\geq 1}\MH_n$ if the critical point $1-\lambda$ is eventually mapped to $1$.

\begin{lema}\label{lema-no-center}
Every hyperbolic component in $\MM_d$ and capture domain in $\MH_n$ has a center, where $n\geq 1$. Meanwhile, $\MH_0$ has no center.
\end{lema}

It will be proved in next section that every hyperbolic component in $\MM_d$ and capture domain in $\MH_n$ has exactly one center, where $n\geq 1$ (Theorem \ref{thm-H_n-simp}).

\begin{proof}
Let $W$ be a hyperbolic component in $\MM_d$. For every $\lambda\in W$, let $m(\lambda)$ be the multiplier of the attracting periodic orbit of $T_{d\lambda}$ other than $1\leftrightarrow\infty$. It can be checked directly that the multiplier mapping $\lambda\mapsto m(\lambda)$ defined from $W$ to $\D$ is proper and holomorphic. This means that $W$ has at least one center.

Let $W$ be a component of $\MH_n$, where $n\geq 1$. Then for every $\lambda\in W$, $T_{d\lambda}^{\circ n}(0)\in\A_{d\lambda}(1)$ and $n$ is smallest. Let $\psi_\lambda:\A_{d\lambda}(1)\rightarrow\D$ be the unique B\"{o}ttcher map define on the immediate basin of $1$ such that $\psi_\lambda\circ U_{d\lambda}=(\psi_\lambda(z))^d$, $\psi_\lambda(1)=0$ and $\psi_\lambda'(1)=1$. By the definition of $\psi_\lambda$, it follows that $\psi_\lambda$ depends holomorphically on $\lambda \in W$. Define a map $m:W\rightarrow \D$ by $m(\lambda)=\psi_\lambda(T_{d\lambda}^{\circ n}(0))$. It is clearly that $m$ is holomorphic. We then prove $m$ is proper. Let $\lambda_k\in H$ be a sequence converging to $\lambda\in \partial W$ as $n\rightarrow \infty$. Suppose that there exists a subsequence of $\lambda_k$, denote also by $\lambda_k$, such that $m(\lambda_k)$ converges to an interior point $w\in\D$. Since the family of univalent mappings $\{\psi_{\lambda_k}^{-1}:\D\rightarrow\C\}$ is normal, we can suppose that $\psi_{\lambda_k}^{-1}\rightarrow \psi^{-1}$ locally uniformly on $\D$. So $\psi^{-1}(\D)\subset\A_{d\lambda}(1)$. This means that $\psi^{-1}(w)=\lim_{k\rightarrow\infty}\psi_{\lambda_k}^{-1}(m(\lambda_k))=\lim_{k\rightarrow\infty}T_{d\lambda_k}^{\circ n}(0)=T_{d\lambda}^{\circ n}(0)\in\A_{d\lambda}(1)$. This contradicts $T_{d\lambda}^{\circ n}(0)\in J_{d\lambda}$ since $\lambda\in\partial W$.

Finally, by the definition of $\MH_0$ and Lemma \ref{equiv-condi}, $\A_{d\lambda}(1)$ contains only one critical point $1$ (counted without multiplicity). Note that $\A_{d\lambda}(1)$ lies in a superattracting periodic Fatou component and $T_{d\lambda}(1-\lambda)=0\neq 1$, it follows that the orbit of $1-\lambda$ is disjoint with the orbit $1\leftrightarrow\infty$. The proof is complete.
\end{proof}

Now we give a complete characterization of the quasiconformal conjugacy classes in $\mathcal{R}_d$.

\begin{thm}\label{thm-qua-class}
Quasiconformal conjugacy classes in $\mathcal{R}_d$ can be listed as follows:

$(1)$ Hyperbolic components in the interior of $\MM_d$ with the center removed.

$(2)$ Capture components of $\MH_n$ with the center (if any) removed, where $n\geq 0$.

$(3)$ Centers of hyperbolic or capture domains.

$(4)$ Queen components in the interior of $\MM_d$.

$(5)$ Single points on the boundary of $\MM_d$.
\end{thm}

\begin{proof}
By Corollary \ref{cor-quasi-rid}, the five cases stated in the theorem are disjoint to each other and (4)(5) are indeed quasiconformal conjugacy classes. (1)(2) are quasiconformal conjugacy classes by Theorem \ref{thm-const-crit}. As every queer component is a conjugacy class, one can get a proof in \cite[Theorem 3.4]{Za} by a word for word analysis.
\end{proof}

\section{Simply connectivity of the capture domains}\label{sec-simp-con}

In this section, we prove that the non-escaping locus $\MM_d$ is connected. This amounts to showing that $\MH_0$ is homeomorphic to the punctured disk $\D^*:=\D\setminus\{0\}$ and each of the component of $\MH_n$ is homeomorphic to the unit disk for $n\geq 1$.

One way to do this is to follow the standard way of Douady-Hubbard's parameterization of the hyperbolic components of the quadratic Mandelbrot set \cite{Do}. This method was developed by Roesch to study the parameter space of the cubic Newton maps \cite{Ro1,Ro2} and Qiu, Roesch, Wang and Yin to study the parameter space of the McMullen maps \cite{QRWY}. Moreover, this parameterized method was generated and then used in the proof of $\MM_2$ is connected \cite[Theorem 1.1]{WQYQG}.

However, to prove $\MH_0$ is homeomorphic to the punctured disk $\D^*$ and each of the component of $\MH_n$ is homeomorphic to the unit disk for $n\geq 1$, it would be much easier to use the methods of Teichm\"{u}ller theory of the rational maps which was developed in \cite{McS} (in which, a different proof of the connectivity of the Mandelbrot set was given).

We first recall some definitions in \cite{McS}.
By definition, the \emph{Teichm\"{u}ller space} $\Teich(T_{d\lambda})$ of $T_{d\lambda}$ consists of all pairs $(T_{d\lambda'},[\varphi])$, where $\varphi:\EC\rightarrow\EC$ is a quasiconformal mapping which conjugates $T_{d\lambda'}$ to $T_{d\lambda}$. Here $[\varphi]$ means the isotopy class of $\varphi$. The \emph{modular group} $\Mod(T_{d\lambda})$ is the group of isotopy classes of quasiconformal homeomorphism commuting with $T_{d\lambda}$. The modular group $\Mod(T_{d\lambda})$ acts on the Teichm\"{u}ller space $\Teich(T_{d\lambda})$ properly discontinuously by $[\psi](T_{d\lambda'},[\varphi])=(T_{d\lambda'},[\psi\circ\varphi])$. The \emph{moduli space} of $T_{d\lambda}$ is defined as the quotient $\Teich(T_{d\lambda})/\Mod(T_{d\lambda})$, which is isomorphic to the quasiconformal conjugacy class of $T_{d\lambda}$.

Moreover, one can define the Teichm\"{u}ller space $\Teich(U,T_{d\lambda})$ on an open set $U$ which is invariant under $T_{d\lambda}$. The set $\Teich(U,T_{d\lambda})$ consists of all the triples $(V,T_{d\lambda'},[\varphi])$, where $V$ is open and invariant under $T_{d\lambda'}$, and the quasiconformal mapping $\varphi:V\rightarrow U$ conjugates $T_{d\lambda'}$ to $T_{d\lambda}$. Here $[\varphi]$ denotes the isotopy class of $\varphi$ relative ideal boundary of $V$.

\begin{thm}\label{thm-H_n-simp}
Each component of $\MH_n$ is homeomorphic to $\D$ and contains exactly one center, where $n\geq 1$. Moreover, $\MH_0$ is homeomorphic to the punctured disk $\D^*$.
\end{thm}

\begin{proof}
Let $W$ be a component of $\MH_n$ with all centers removed. Then the forward orbit of $1-\lambda$ under $T_{d\lambda}$ is infinite for $\lambda\in W$. By Theorem \ref{thm-qua-class}, $W$ denotes a single quasiconformal conjugacy class.

For any basepoint $\lambda\in W$, it follows that the critical point $1-\lambda$ belongs to the attracting basin of the cycle $1\mapsto\infty\mapsto 1$. In particular, $T_{d\lambda}^{\circ n}(0)\in\A_{d\lambda}(1)$ and $T_{d\lambda}^{\circ n}(0) \neq 1$. Define the Green function on $\A_{d\lambda}(1)$ by
\begin{equation*}
G_{d\lambda}(z)=-\lim_{k\rightarrow\infty}d^{-k}\log|U_{d\lambda}^{\circ k}(z)-1|, \text{~where~} z\in \A_{d\lambda}(1).
\end{equation*}
Note that $G_{d\lambda}$ can be extended to the Fatou set of $T_{d\lambda}$ by pulling back.

Let $\gamma$ be the equipotential of $G_{d\lambda}$ passing through $1-\lambda$. Then $\gamma$ is homeomorphic to the figure 8. Define
\begin{equation*}
\widehat{J}_{d\lambda}:=J_{d\lambda}\cup\bigcup_{n\in\Z}T_{d\lambda}^{\circ n}(\gamma\cup\{0\}).
\end{equation*}
Then $\widehat{J}_{d\lambda}$ is the closure of the grand orbits of all periodic points and critical points of $T_{d\lambda}$. The complement $U:=\EC\setminus\widehat{J}_{d\lambda}$ consists of countably many annuli with finite modulus which lie in a same grand orbit. By \cite[Theorem 6.2]{McS}, we have
\begin{equation*}
\Teich(T_{d\lambda})\simeq \Teich(U,T_{d\lambda})\times M_1(J_{d\lambda},T_{d\lambda}),
\end{equation*}
where $M_1(J_{d\lambda},T_{d\lambda})$ denotes the unit ball in the space of all $T_{d\lambda}$-invariant Beltrami differentials supported on $J_{d\lambda}$. Note that every hyperbolic rational map carries no invariant line fields on the Julia set, it follows that $M_1(J_{d\lambda},T_{d\lambda})$ is trivial since $T_{d\lambda}$ is hyperbolic when $\lambda\in W\subset\MH_n$.

Since $W$ denotes a single quasiconformal conjugacy class, we have
\begin{equation*}
W\simeq \Teich(T_{d\lambda})/\Mod(T_{d\lambda})\simeq \Teich(U,T_{d\lambda})/\Mod(T_{d\lambda})\simeq \mathbb{H}/\Mod(T_{d\lambda})
\end{equation*}
by \cite[Theorem 6.1]{McS}. Note that every quasiconformal self-conjugacy $\psi$ of $T_{d\lambda}$ fixes the grand orbits of the critical points $1$ and $1-\lambda$ and hence fixes the boundaries of each annulus of $U$. Moreover, $\psi$ is the identity on $J_{d\lambda}$. Therefore, $[\psi]\in\Mod(T_{d\lambda})$ is identity on $\widehat{J}_{d\lambda}$ and it is possibly a power of a Dehn twist in the annuli of $U$. This means that $\Mod(T_{d\lambda})$ is a subgroup of $\Z$.

By Lemma \ref{lema-no-center}, each $W$ cannot be simply connected is a component of $\MH_n$ for $n\geq 1$. On the other hand, $W$ is not simply connected if $W=\MH_0$ by Proposition \ref{prop-H-0-unbound}. So $\Mod(T_{d\lambda})=\Z$. This means that $W$ is homeomorphic to a punctured disk. This means that each $W$ contains exactly only one center if $W\neq\MH_0$. The proof is complete.
\end{proof}

\vskip0.2cm
\noindent\emph{Proof of Theorem \ref{M-connected}}. This is a direct corollary of Proposition \ref{prop-decp} and Theorem \ref{thm-H_n-simp}.
\hfill $\square$

\section{Proof of the asymptotic formula}\label{pf-asym-fml}

By Proposition \ref{prop-H-0-unbound}, if the parameter $\lambda$ lies in the unbounded capture domain $\MH_0$, then the Julia set $J_{d\lambda}$ is a quasicircle. In this case, $J_{d\lambda}$ moves holomorphically in $\MH_0$ and its Hausdorff dimension depends real analytically on $\lambda$ by a classic result of Ruelle. The following Theorem \ref{Ruelle} is a weak version of \cite[Corollary 6]{Ru}.

\begin{thm}\label{Ruelle}
Let $f_\lambda:\Lambda\times\EC\rightarrow\EC$ be a holomorphic family of hyperbolic rational maps parameterized by $\Lambda$, where $\Lambda$ is a complex manifold. Then the Hausdorff dimension of the Julia set of $f_\lambda$ depends real analytically on $\lambda\in\Lambda$.
\end{thm}

Let $\Omega$ be a closed subset of $\R^n$. A map $S:\Omega\rightarrow \Omega$ is called a \emph{contraction} on $\Omega$ if there exists a real number $c\in(0,1)$ such that $|S(x)-S(y)| \leq c|x-y|$ for all $x, y\in\Omega$. A finite family of contractions $\{S_1, S_2,\cdots, S_m\}$ defined on $\Omega\subset\R^n$, with $m\geq 2$, is called an \emph{iterated function system} or IFS in short.

To compute the Hausdorff dimension of $J_{d\lambda}$ with $\lambda\in\MH_0$, we need the following result (see \cite[Theorem 9.1, Propositions 9.6 and 9.7]{Fa}).

\begin{thm}[{\cite{Fa}}]\label{Falconer}
Let $\{S_1,\ldots,S_m\}$ be an IFS on a closed set $\Omega\subset\mathbb{R}^n$ such that $|S_i(x)-S_i(y)|\leq c_i|x-y|$ with $0<c_i<1$. Then:

$(1)$ There exists a unique non-empty compact set $J$ such that $J=\bigcup_{\,i=1}^{\,m} S_i(J)$.

$(2)$ The Hausdorff dimension $\dim_H(J)$ of $J$ satisfies $\dim_H(J)\leq s$, where $\sum_{i=1}^mc_i^s=1$.

$(3)$ If we require further $|S_i(x)-S_i(y)|\geq b_i|x-y|$ for $0<b_i<1$, then $\dim_H(J)\geq s'$, where $\sum_{i=1}^m b_i^{s'}=1$.
\end{thm}

The non-empty compact set $J$ appeared in Theorem \ref{Falconer}(1) is called the \emph{attractor} of the IFS $\{S_1,\ldots,S_m\}$.

Let $f$ be a rational map with degree at least two. We use $\text{Fix}(f)$ to denote the set of all the fixed points in the Julia set of $f$.

\begin{lema}\label{lema-fix}
Let $f$ be a hyperbolic rational map whose Julia set $J$ is a quasicircle. Then the Hausdorff dimension $D:=\dim_H(J)$ of $J$ is determined by $A_n(D)=\mathcal{O}(1)$ as $n\rightarrow\infty$, where
\begin{equation}\label{O(1)=}
A_n(D)=\sum_{z\in\,\textup{Fix}(f^{\circ n})}|(f^{\circ n})'(z)|^{-D}.
\end{equation}
\end{lema}

Under the assumption of Lemma \ref{lema-fix}, $\text{Fix}(f^{\circ n})$ denotes the collection of all the repelling periodic points of $f$ with period exactly $n$. The Julia set of a hyperbolic rational map can be seen as the limit of a sequence of IFS. These IFS are defined in terms of the inverse branches of the iterations of the rational map. The original proof idea of Lemma \ref{lema-fix} comes from \cite[Lemma 2.6]{YW} and the proof appeared here is an improved version.

\begin{proof}
Let $d\geq 2$ be the degree of $f$. Since $f$ is hyperbolic and the Julia set $J$ of $f$ is a quasicircle, there exist a pair of closed annular neighborhoods $W_1,W_2$ of $J$ and a quasiconformal mapping $\phi:W_1\rightarrow\mathbb{A}_\varepsilon$, such that $\phi$ conjugates $f:W_1\rightarrow W_2$ to $z\mapsto z^d$ or $z\mapsto z^{-d}$, where $\mathbb{A}_\varepsilon:=\{z:1-\varepsilon\leq |z|\leq 1+\varepsilon\}$ is a closed annular neighborhood of the unit circle and $\varepsilon>0$ is small enough.
Without loss of generality, we only consider the first case since the completely similar argument can be applied to the second one.

In order to define IFS, it is more convenient to lift $J$ and $f$ under the exponential map. Hence we assume further that $J$ separates $0$ and $\infty$. Define a curve $\gamma:=\phi^{-1}([(1-\varepsilon)^d,(1+\varepsilon)^d])\subset W_2$. Fix a component of $\exp^{-1}(W_2\setminus\gamma)$ and denote it by $U$. Then $U$ is topologically a strip and $\exp:U\rightarrow W_2\setminus\gamma$ is conformal in the interior of $U$, whose inverse is denoted by $\log:W_2\setminus\gamma\rightarrow U$ (see Figure \ref{Fig_lift}).

\begin{figure}[!htpb]
  \setlength{\unitlength}{1mm}
  \centering
  \includegraphics[width=140mm]{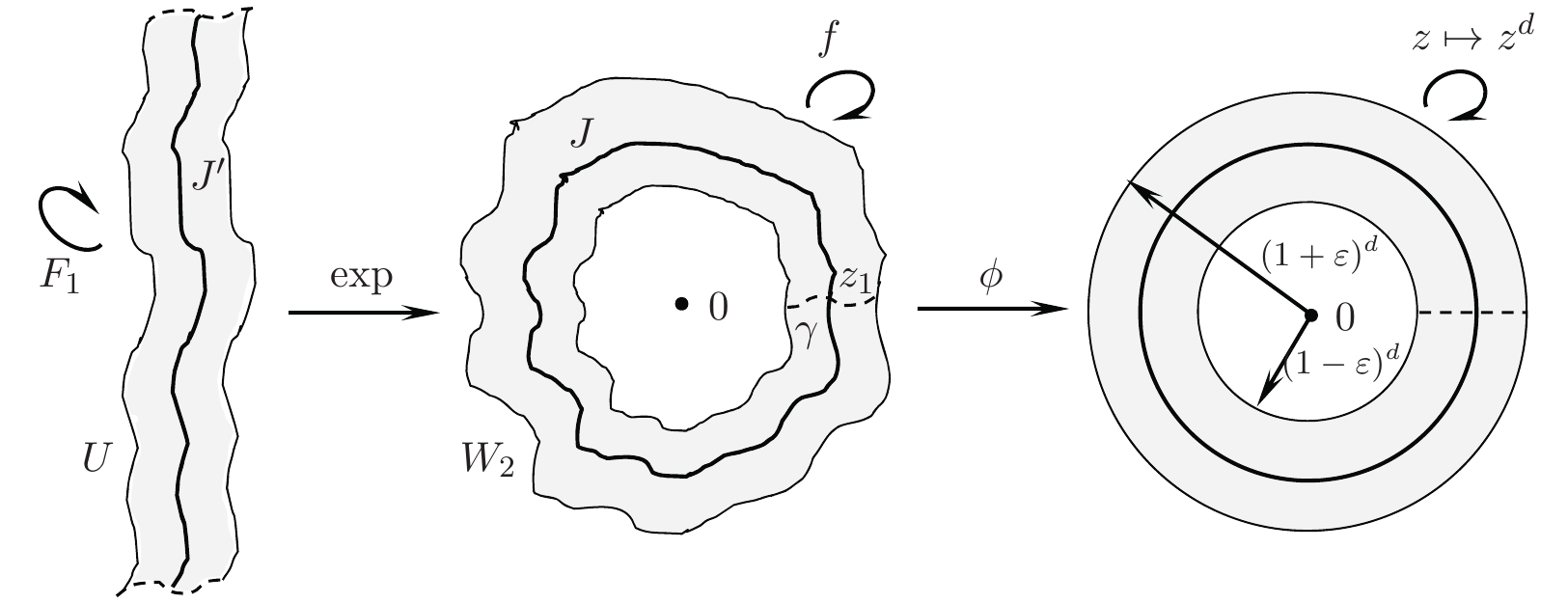}
  \caption{Sketch illustration of the construction of the IFS.}
  \label{Fig_lift}
\end{figure}

For each $n\geq 1$, the map $f^{\circ n}:W_1\rightarrow W_2$ has $d^n$ inverse branches, say $T_1,\cdots,T_{d^n}$, each maps $W_2\setminus\gamma$ onto a half open quadrilateral such that their images are arranged in anticlockwise order one by one. Let $S_i:=\log\circ T_i\circ\exp$ be the map defined in $U$, where $1\leq i\leq d^n$. It is easy to see each $S_i$ is conformal in the interior of $U$ and can be conformally extended to an open neighborhood of $\overline{U}$.

Now it is easy to see $\{S_1,\cdots,S_{d^n}\}$ is an IFS defined on $\overline{U}$ since $f$ is strictly expanding on $W_1$. The attractor $J'$ of $\{S_1,\cdots,S_{d^n}\}$ is a closed set satisfying $J=\exp(J')$. Moreover, $J\setminus\{z_1\}$ is the conformal image of $J'$ with two ends removed, where $z_1\in J\cap\gamma$ is a fixed point of $f$. This means that the Hausdorff dimensions of $J'$ and $J$ satisfy $\dim_H(J')=\dim_H(J)$.

Let $F_n|_U:=\bigsqcup_{i=1}^{d^n}S_i^{-1}|_{S_i(U)}$ be the lift of $f^{\circ n}$ under $\exp$. Then each $S_i(U)$ contains exactly one fixed point $\zeta_i\in J'$ of $F_n$ in its interior for $1<i<d^n$ and on its boundary for $i=1$ and $d^n$. Since $S_i$ can be conformally extended to an open neighborhood of $\overline{U}$, by Koebe's distortion theorem, there exist two constants $0<C_1\leq 1\leq C_2$ both independent of $n$, such that
\begin{equation*}
\frac{C_1}{|F_n'(\zeta_i)|}\leq\frac{|S_i(x)-S_i(y)|}{|x-y|}
\leq\frac{C_2}{|F_n'(\zeta_i)|},~\forall\,1\leq i\leq d^n,\  x,y\in \overline{U}.
\end{equation*}

By Theorem \ref{Falconer}, the Hausdorff dimension $D=\dim_H(J')=\dim_H(J)$ satisfies $s_1\leq D\leq s_2$, where $\sum_{i=0}^{d^n} C_j^{s_j}|F_n'(\zeta_i)|^{-s_j}=1$ and $j=1,2$. Then, we have
\begin{equation}\label{inequality}
\frac{1}{C_2^D}\leq \frac{1}{C_2^{s_2}}\leq \sum_{i=1}^{d^n}\frac{1}{|F_n'(\zeta_i)|^{s_2}}
\leq \sum_{i=1}^{d^n}\frac{1}{|F_n'(\zeta_i)|^{D}}\leq \sum_{i=1}^{d^n}\frac{1}{|F_n'(\zeta_i)|^{s_1}}
=\frac{1}{C_1^{s_1}}\leq \frac{1}{C_1^D}.
\end{equation}
The $d^n-1$ fixed points of $f^{\circ n}$ in the Julia set $J$ are $\{z_i=\exp(\zeta_i):1\leq i<d^n\}$. In particular, $z_1=\exp(\zeta_1)=\exp(\zeta_{d^n})$. Since $F_n$ is conformally conjugate to $f^{\circ n}$ in the interior of each $S_i(U)$, we have $F_n'(\zeta_i)=(f^{\circ n})'(z_i)$ for $1\leq i<d^n$. Therefore, by \eqref{inequality}, we have
\begin{equation*}
 \sum_{z\in\,\textup{Fix}(f^{\circ n})}\frac{1}{|(f^{\circ n})'(z)|^D}
=\sum_{i=1}^{d^n}\frac{1}{|(f^{\circ n})'(z_i)|^{D}}
=\sum_{i=1}^{d^n}\frac{1}{|F_n'(\zeta_i)|^{D}}-|F_n'(\zeta_{d^n})|^{-D}=\mathcal{O}(1)
\end{equation*}
since $|F_n'(\zeta_{d^n})|\rightarrow\infty$ as $n\rightarrow\infty$. The proof is complete.
\end{proof}

As the parameter $\lambda$ tends to $\infty$, the diameter of the Julia set $J_{d\lambda}$ of $T_{d\lambda}$ becomes larger and larger in the Euclidean metric and the shape of $J_{d\lambda}$ becomes more and more circular (see Figure \ref{Fig_quasicircle}). Therefore, one can make a scaling of $J_{d\lambda}$ (or equivalently, make a conjugate), such the new Julia set converges to the unit circle.
\begin{figure}[!htpb]
  \setlength{\unitlength}{1mm}
  \centering
  \includegraphics[width=70mm]{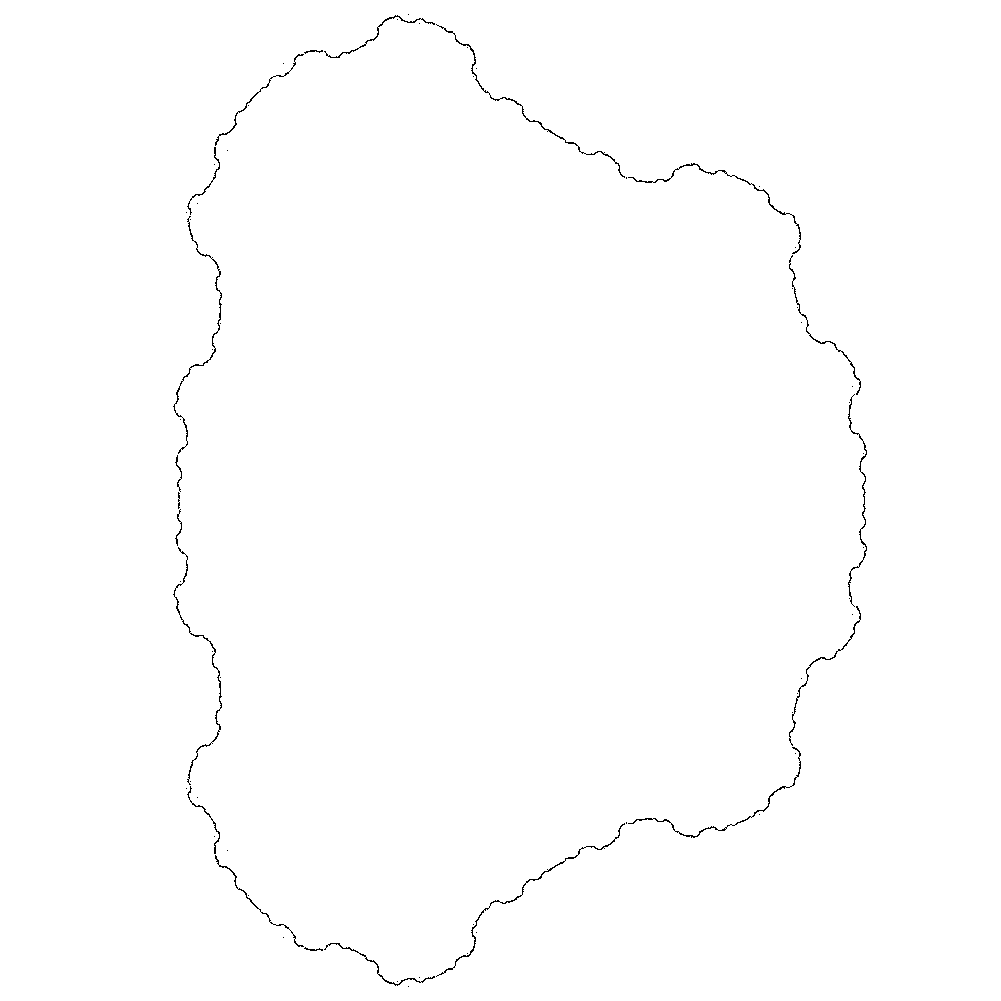}
  \includegraphics[width=70mm]{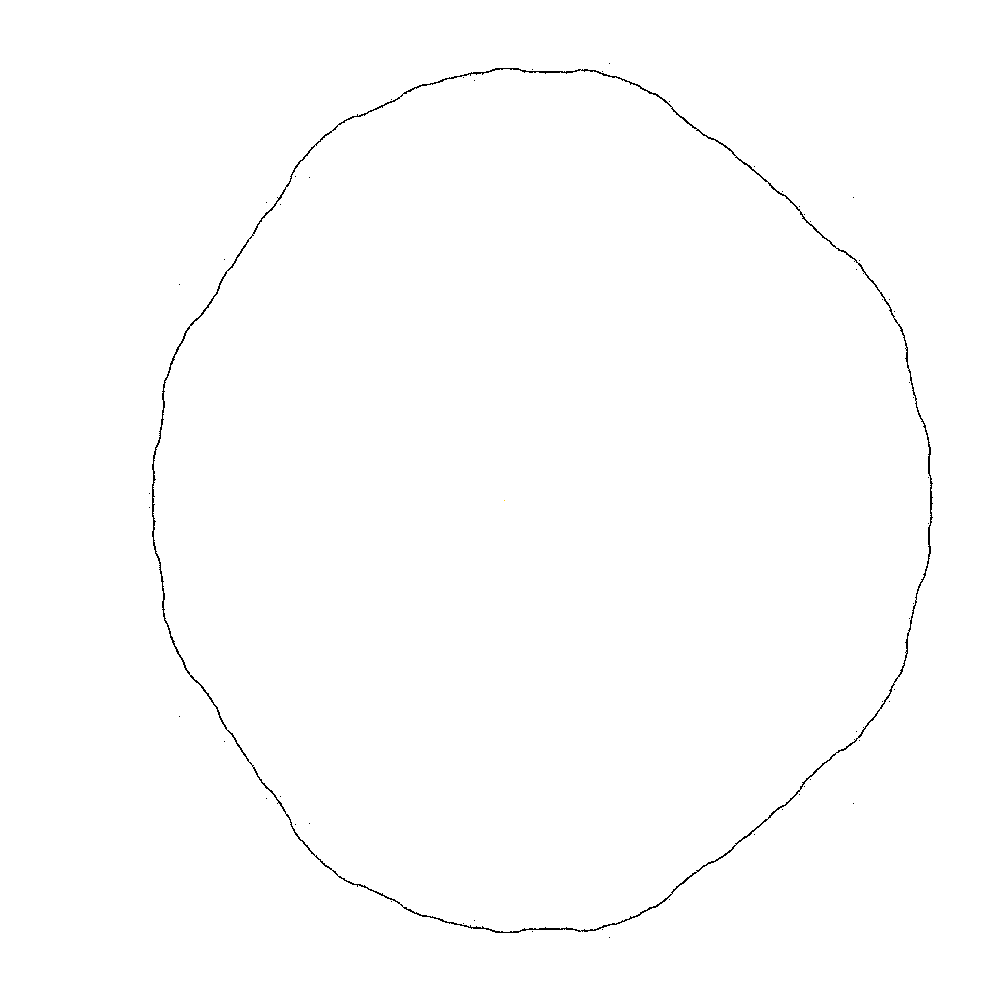}
  \caption{The Julia sets of $T_{2\lambda}$, both are quasicircles, where $\lambda=30$ and $1000$, respectively. It can be seen that the Julia set becomes more circular as the parameter $\lambda$ becomes more larger (compare the right picture in Figure \ref{Fig_Julia-Siegel}). Figure ranges: $[-10,16]\times[-13,13]$ and $[-125,125]\times[-125,125]$.}
  \label{Fig_quasicircle}
\end{figure}

Specifically, define
\begin{equation}
J_{d\lambda}^*=\{\lambda^{-\frac{d}{d+1}}(z-1) : z\in J_{d\lambda}\}.
\end{equation}
The following Lemma \ref{limit-mcm} has been proved in \cite[Theorem 4.3]{Qi} as a special case.

\begin{lema}\label{limit-mcm}
The scaled Julia set $J_{d\lambda}^*$ converges to the unit circle in the Hausdorff topology as $\lambda$ tends to $\infty$ and the Hausdorff dimension of $J_{d\lambda}$ tends to $1$ as $\lambda$ tends to $\infty$.
\end{lema}

Although Lemma \ref{limit-mcm} is significant, however, we want to know further about the asymptotic formula of of the Hausdorff dimension of $J_{d\lambda}$ as $\lambda$ tends to $\infty$. In order to calculate the Hausdorff dimension of $J_{d\lambda}$, we do some setting first.

Recall that in Proposition \ref{prop-H-0-unbound}, $\alpha=\lambda^{-\frac{1}{d+1}}$. Then $\lambda\alpha^d=\alpha^{-1}$. Let $\varphi_\alpha(z)=\alpha^d(z-1)$ be the linear transformation as before. We define a new rational map with parameter $\alpha$ as
\begin{equation}\label{f-beta}
f_\alpha(z):= \varphi_\alpha\circ T_{d\lambda}\circ\varphi_\alpha^{-1}=\sum_{i=0}^{d-1}\frac{C_d^i\,\alpha^i}{z^{d-i}}
=\frac{1}{z^d}+\frac{C_d^1\alpha}{z^{d-1}}+\cdots+\frac{C_d^1\alpha^{d-1}}{z}.
\end{equation}
This means that there exists a small $\varepsilon>0$ such that $f_\alpha:\D_\varepsilon\times\EC\rightarrow\EC$ is a holomorphic family of hyperbolic rational maps parameterized by $\D_\varepsilon$, where $\D_\varepsilon:=\{z:|z|<\varepsilon\}$.
Note that the Hausdorff dimension is invariant under a conformal isomorphism. This means that we only need to calculate the Hausdorff dimension of the Julia set $J_\alpha$ of $f_\alpha$ with $\alpha\in\D_\varepsilon$ since $\dim_H(J_\alpha)=\dim_H(J_{d\lambda})$. We would like to remark that $J_\alpha=J_{d\lambda}^*$.

Let $E$ be a subset of $\EC$ and $(\Lambda,\lambda_0)$ a connected complex manifold with basepoint $\lambda_0$. A family of maps $h_\lambda:E\rightarrow\EC$ is called a \textit{holomorphic motion} of $E$ parameterized by $\Lambda$ and with base point $\lambda_0$ if: (1) For each $\lambda\in\Lambda$, $h_\lambda$ is injective on $E$; (2) For each $z\in E$, $h_\lambda(z)$ is a holomorphic function of $\lambda\in\Lambda$; and (3) $h_{\lambda_0}$ is identity on $E$ (see \cite{Ly}, \cite{MSS} or \cite[Chap.\,4]{Mc}).

\vskip0.2cm
\noindent\textit{Proof of Theorem \ref{thm-Haus}.}
By \eqref{f-beta}, it follows that the Julia set $J_\alpha$ is the unit circle if $\alpha=0$.
For $z\in J_0=\T$, we have $f_0(z)=z^{-d}$. Note that $f_\alpha$ is a holomorphic family of hyperbolic rational maps with parameter $\alpha\in\D_\varepsilon$. There exists a holomorphic motion $\phi_\alpha:J_0\rightarrow\EC$ of $J_0$ parameterized by $\D_\varepsilon$ and with base point $0$ such that $\phi_\alpha(J_0)=J_\alpha$ and
\begin{equation}\label{conjugation}
f_\alpha\circ\phi_\alpha(z)=\phi_\alpha\circ f_0(z)=\phi_\alpha(z^{-d})
\end{equation}
for all $z\in J_0$, see \cite[Chap. 4]{Mc}. Since every point on $J_0$ moves holomorphically, we can write $\phi_\alpha(z)$ in power series of $\alpha$ as
\begin{equation}\label{z(t)}
\phi_\alpha(z)=z\,(1+u_1(z)\alpha+u_2(z)\alpha^2+\mathcal{O}(\alpha^3)),
\end{equation}
where $z\in J_0$.

In the following, we adopt the notation $q:=-d$ since the negative sign is boring in the expressions during the calculation. Meantime, we assume that $d\geq 3$ first. If $\alpha$ is small enough, we can expand $f_\alpha$ in \eqref{f-beta} in power series of $\alpha$ as
\begin{equation}\label{f-beta-expand}
f_\alpha(z)= z^q-qz^{q+1}\alpha+\frac{q(q+1)}{2}z^{q+2}\alpha^2+\mathcal{O}(\alpha^3).
\end{equation}

Substituting \eqref{z(t)} and \eqref{f-beta-expand} into \eqref{conjugation}, then comparing the terms to the second order in $\alpha$, we obtain the following equations:
\begin{eqnarray}
u_1(z^q)-qu_1(z) &=& -qz,  \label{equation_1}\\
u_2(z^q)-qu_2(z) &=& \frac{q(q-1)}{2}u_1^2(z)-q(q+1)zu_1(z)+\frac{q(q+1)}{2}z^2. \label{equation_2}
\end{eqnarray}
For each non-zero integer $l\in\Z$, the functional equation
\begin{equation}\label{fun-equ}
u(z^q)-qu(z) = -q z^l
\end{equation}
has the formal solution
\begin{equation}\label{solution}
u(z)=\sum_{k=0}^{+\infty}\frac{z^{lq^{k}}}{q^k}.
\end{equation}
Note that the solution \eqref{solution} is convergent if $|z|\leq 1$. This means that the solution of \eqref{equation_1} is
\begin{equation}\label{solution_u1}
u_1(z)=\sum_{k=0}^{+\infty}\frac{z^{q^k}}{q^k}.
\end{equation}
Therefore, the equation \eqref{equation_2} can be reduced to
\begin{equation}
u_2(z^q)-qu_2(z) = -q \left( (q+1)\sum_{l=0}^{+\infty}\frac{z^{q^l+1}}{q^l}
-\frac{q-1}{2}\left(\sum_{l=0}^{+\infty}\frac{z^{q^l}}{q^l}\right)^2 -\frac{q+1}{2}z^2\right).
\end{equation}
By \eqref{fun-equ} and \eqref{solution}, the solution of $u_2$ is
\begin{equation}\label{solution_u2}
u_2(z)=\sum_{k=0}^{+\infty}\left((q+1)\sum_{l=0}^{+\infty}\frac{z^{q^{l+k}+q^{k}}}{q^{l+k}}
         -\frac{(q-1)}{2q^{k}}\left(\sum_{l=0}^{+\infty}\frac{z^{q^{l+k}}}{q^{l}}\right)^2
         -\frac{(q+1)}{2q^{k}}z^{2q^{k}}\right).
\end{equation}

For each $n\geq 1$, the collection of the fixed points of $f_\alpha^{\circ n}$ on the Julia set $J_\alpha$ forms the finite set
\begin{equation}\label{fixed_point}
\text{Fix}(f_\alpha^{\circ n})=\left\{\phi_\alpha(e^{2\pi it_j}):t_j=\frac{j}{q^n-1},1\leq j\leq |q^n-1|\right\}.
\end{equation}
By \eqref{conjugation} and the chain rule, we have
$(f_\alpha^{\circ n})'(\phi_\alpha(e^{2\pi it_j}))=\prod_{m=0}^{n-1}f'_\alpha(\phi_\alpha(e^{2\pi iq^m t_j}))$.
The calculation in Appendix (\S\ref{sec-appendix}) shows that for every $D>0$ and all sufficiently large $n$, the following holds:
\begin{equation}\label{appendix}
\frac{1}{|q^n-1|}\sum_{j=1}^{|q^n-1|}\prod_{m=0}^{n-1} \left|{f'_\alpha}(\phi_\alpha(e^{2\pi i q^m t_j}))\right|^{-D}
=|q|^{-nD}\left(1+\frac{D^2n}{4}|\alpha|^2+\mathcal{O}(\alpha^3)\right).
\end{equation}
Let $D_\alpha:=\dim_H(J_\alpha)$ be the Hausdorff dimension of $J_\alpha$. One can write the corresponding \eqref{O(1)=} of $f_\alpha$ in Lemma \ref{lema-fix} as
\begin{equation}\label{q^n-1}
|q^n-1|\,|q|^{-nD_\alpha}\left(1+\frac{D_\alpha^2n}{4}|\alpha|^2+\mathcal{O}(\alpha^3)\right) =\mathcal{O}(1).
\end{equation}

Fix some large $n$, when $\alpha$ is small enough, \eqref{q^n-1} is equivalent to
\begin{equation}\label{exp}
\exp{\left(n\left(\frac{D_\alpha^2}{4}|\alpha|^2-(D_\alpha-1)\log{|q|}\right)
+\mathcal{O}(\alpha^3)\right)}=\mathcal{O}(1).
\end{equation}
By Theorem \ref{Ruelle} and Lemma \ref{limit-mcm}, $D_\alpha$ depends real analytically on $\alpha$ in a small neighborhood of the origin and $D_0=1$. This means that in a small neighborhood of 0, the Hausdorff dimension of $J_\alpha$ can be written as
\begin{equation}\label{finish-last}
D_\alpha=1+a_{10}\alpha+a_{01}\overline{\alpha}+a_{20}\alpha^2+a_{02}\overline{\alpha}^2+a_{11}|\alpha|^2+\mathcal{O}(\alpha^3).
\end{equation}
Substituting \eqref{finish-last} into \eqref{exp} and comparing the corresponding coefficients, we have
\begin{equation}
a_{10}=a_{01}=a_{20}=a_{02}=0\text{~and~} a_{11}=1/(4\log|q|).
\end{equation}
This means that
\begin{equation}\label{finish}
D_\alpha=1+\frac{|\alpha|^2}{4\log{|q|}}+\mathcal{O}(\alpha^3).
\end{equation}
Note that $q=-d$ and $\alpha=\lambda^{-\frac{1}{d+1}}$. This ends the proof of Theorem \ref{thm-Haus} in the case of $d\geq 3$.

If $d=2$, then \eqref{f-beta-expand} can be written as $f_\alpha(z)= z^q-qz^{q+1}\alpha$. Following the calculation process of $d\geq 3$ and carefully omitting some corresponding terms, it can be checked that Theorem \ref{thm-Haus} still holds for $d=2$. The proof is complete.
\hfill $\square$

\section{Appendix}\label{sec-appendix}

This section will devote to proving \eqref{appendix}. From \eqref{f-beta-expand}, we have
\begin{equation}\label{f-beta-deri}
f_\alpha'(z)= q z^{q-1}-q(q+1)z^q\alpha+\frac{q(q+1)(q+2)}{2}z^{q+1}\alpha^2+\mathcal{O}(\alpha^3).
\end{equation}
Substituting \eqref{z(t)} into \eqref{f-beta-deri}, we have
\begin{equation}\label{f_alpha-de}
\begin{split}
f'_\alpha(\phi_\alpha(z))=
&~q{z}^{q-1}+ q z^{q-1}[(q-1)u_1(z)-(q+1)z]\,\alpha
  + qz^{q-1}\left[ \frac{(q+1)(q+2)}{2}z^2\right.\\
&\left.  +\frac{(q-1)(q-2)}{2}u_1^2(z)-q(q+1)z u_1(z)+(q-1)u_2(z)\right]\alpha^2+\MO(\alpha^3).
\end{split}
\end{equation}

Define $\sigma:=\sigma(t)=e^{2\pi i t}\in\T$. Then $\sigma\overline{\sigma}=1$. For $0\leq m\leq n-1$, by \eqref{f_alpha-de}, we have
\begin{equation}\label{f_alpha-de-m}
\begin{split}
 &~ |f'_\alpha(\phi_\alpha(\sigma^{q^m}))|^2
    = f'_\alpha(\phi_\alpha(\sigma^{q^m}))\,\overline{f'_\alpha(\phi_\alpha(\sigma^{q^m}))} \\
=&~ q^2+A_m\alpha+\overline{A}_m\overline{\alpha}+A_m\overline{A}_m|\alpha|^2/q^2
    +B_m\alpha^2+\overline{B}_m\overline{\alpha}^2+\mathcal{O}(\alpha^3),
\end{split}
\end{equation}
where
\begin{equation}\label{A_m}
A_m =  q^2(q-1)\,u_1(\sigma^{q^m})-q^2(q+1)\,\sigma^{q^m}
\end{equation}
and
\begin{equation}\label{B_m}
\begin{split}
B_m = &~ \frac{q^2(q+1)(q+2)}{2}{\sigma}^{2q^m} + \frac{q^2(q-1)(q-2)}{2}u_1^2(\sigma^{q^m})\\
      &~ -q^3(q+1){\sigma}^{q^m}u_1(\sigma^{q^m}) + q^2(q-1)u_2(\sigma^{q^m}).
\end{split}
\end{equation}

For every $D>0$, by \eqref{f_alpha-de-m}, we have
\begin{equation}\label{times-f'}
\begin{split}
& \prod_{m=0}^{n-1}|f'_\alpha(\phi_\alpha(\sigma^{q^m}))|^{-D}
   =\prod_{m=0}^{n-1}(|f'_\alpha(\phi_\alpha(\sigma^{q^m}))|^2)^{-\frac{D}{2}} \\
& =|q|^{-nD}\prod_{m=0}^{n-1}
   \left(1+\frac{A_m\alpha+\overline{A}_m\overline{\alpha}+B_m\alpha^2+\overline{B}_m\overline{\alpha}^2}{q^2}
  +\frac{A_m\overline{A}_m|\alpha|^2}{q^4}+\MO(\alpha^3)\right)^{-\frac{D}{2}}\\
& = |q|^{-nD}-\frac{D}{2}|q|^{-nD-2}\sum_{m=0}^{n-1}\left(A_m\alpha+\overline{A}_m\overline{\alpha}+
   B_m\alpha^2+\overline{B}_m\overline{\alpha}^2\right)\\
&  -\frac{D}{2}|q|^{-nD-4}\left(\sum_{0\leq m_1<m_2\leq n-1}(A_{m_1}A_{m_2}\alpha^2+\overline{A}_{m_1}
   \overline{A}_{m_2}\overline{\alpha}^2)
   + \sum_{0\leq m_1,m_2\leq n-1}A_{m_1}\overline{A}_{m_2}|\alpha|^2\right)\\
   &+\frac{D(D+2)}{8}|q|^{-nD-4}\left(\sum_{m=0}^{n-1}(A_m\alpha+\overline{A}_m\overline{\alpha})\right)^2
     +\mathcal{O}(\alpha^3).
\end{split}
\end{equation}

\begin{lema}\label{lemma_Appendix}
Let $m,m_1,m_2\in \mathbb{N}$. If $n\geq 1$, then:

$(1)~q^m\not\equiv 0 \textup{ mod }q^n-1$.

$(2)~q^{m_1}+q^{m_2}\not\equiv 0 \textup{ mod }q^n-1$.

$(3)~q^{m_1}-q^{m_2} \equiv 0 \textup{ mod }q^n-1$ if and only if $m_1-m_2=kn$ for some $k\in\Z$.
\end{lema}
\begin{proof}
Since $(q,q^n-1)=1$, it means that $(q^m,q^n-1)=1$ for $m\geq 0$. Then (1) follows.

To prove (2), it suffices to show that $q^m+1\not\equiv 0 \textup{ mod }q^n-1$ for $m\geq 0$ since $q^n-1$ is relative prime to $q^{m'}$ for $m'\geq 0$ by (1). Set $m=kn+r$, where $k\geq 0$ and $0\leq r \leq n-1$. We have
\begin{equation*}
q^m+1= q^{kn+r}-q^r+q^r+1 \equiv q^r+1 \not\equiv 0 \textup{ mod }q^n-1
\end{equation*}
since $0<|q^r+1|<|q^n-1|$.

The proof of (3) is similar to that of (2). Since $q^n-1$ is relative prime to $q^{m'}$ for $m'\geq 0$, we need to find out the condition on $m$ such that $q^m-1\equiv 0 \textup{ mod }q^n-1$ for fixed $n\geq 1$.
Set $m=kn+r$, where $k\geq 0$ and $0\leq r \leq n-1$. We have
\begin{equation*}
q^m-1= q^{kn+r}-q^r+q^r-1 \equiv q^r-1 \textup{ mod }q^n-1.
\end{equation*}
This means that $q^m-1\equiv 0 \textup{ mod }q^n-1$ if and only if $r=0$ since $|q^r-1|<|q^n-1|$.
\end{proof}

Following \cite[$\S$ 2]{WBKS}, it is convenient to introduce the \emph{average notation}
\begin{equation}\label{average}
\langle G(t)\rangle_n:=\frac{1}{|q^n-1|}\sum_{j=1}^{|q^n-1|}G(t_j),
\end{equation}
where $G$ is a continuous function defined on the interval $[0,1)$ and $t_j=j/(q^n-1)$ is defined in \eqref{fixed_point}.

In order to prove \eqref{appendix}, we only need to prove for every $D>0$ and sufficiently large $n$, the following holds
\begin{equation}\label{appendix-reduce}
\left\langle\prod_{m=0}^{n-1} \left|{f'_\alpha}(\phi_\alpha(\sigma^{q^m}))\right|^{-D}\right\rangle_n
=|q|^{-nD}\left(1+\frac{D^2n}{4}|\alpha|^2+\mathcal{O}(\alpha^3)\right).
\end{equation}

For each $n\geq 1$ and any $k\in\Z$, it is straightforward to verify the average in \eqref{average} has the following useful property:
\begin{equation}\label{ave_property}               
\langle \sigma^k \rangle_n=\langle e^{2\pi i kt} \rangle_n=
\left\{                         
\begin{array}{ll}               
1  &~~~\text{if}~k\equiv 0~\text{mod}~q^n-1, \\             
0  &~~~\text{otherwise.}    
\end{array}                     
\right.                         
\end{equation}

\begin{lema}\label{lema-app-2}
For $0\leq m,m_1,m_2\leq n-1$, we have
$\langle \sigma^{q^m}\rangle_n=0$,
$\langle u_1(\sigma^{q^m})\rangle_n=0$,
$\langle \sigma^{q^{m_1}+q^{m_2}}\rangle_n=0$,
$\langle {\sigma}^{q^{m_1}}u_1(\sigma^{q^{m_2}})\rangle_n=0$,
$\langle u_1(\sigma^{q^{m_1}})u_1(\sigma^{q^{m_2}})\rangle_n=0$
and $\langle u_2(\sigma^{q^m})\rangle_n=0$.
\end{lema}

\begin{proof}
By \eqref{solution_u1} and \eqref{solution_u2}, the average property \eqref{ave_property} and Lemma \ref{lemma_Appendix}(1)(2), the equations stated in the Lemma can be verified directly.
\end{proof}

As an immediate corollary of Lemma \ref{lema-app-2}, from \eqref{A_m} and \eqref{B_m}, we have

\begin{cor}\label{appendix_cor}
$\left\langle A_m\right\rangle_n=\left\langle\overline{A}_m\right\rangle_n=0,
\left\langle B_m\right\rangle_n=\left\langle\overline{B}_m\right\rangle_n=0,
\left\langle A_{m_1}A_{m_2}\right\rangle_n=\left\langle\overline{A}_{m_1}\overline{A}_{m_2}\right\rangle_n=0
$ for $0\leq m,m_1,m_2\leq n-1.$
\end{cor}

By \eqref{times-f'} and Corollary \ref{appendix_cor}, we have
\begin{equation}\label{appendix_ave}
\left\langle\prod_{m=0}^{n-1}|f'_\alpha(\phi_\alpha(\sigma^{q^m}))|^{-D}\right\rangle_n
=|q|^{-nD}\left(1+\frac{D^2}{4}|q|^{-4}\sum_{0\leq m_1,m_2\leq n-1}\langle A_{m_1}\overline{A}_{m_2}\rangle_n |\alpha|^2\right) +\mathcal{O}(\alpha^3).
\end{equation}

By \eqref{A_m} and \eqref{B_m}, we have
\begin{equation}\label{appendix_AA}
\begin{split}
\left\langle A_{m_1}\overline{A}_{m_2}\right\rangle_n=
&~q^4(q-1)^2\langle u_1(\sigma^{q^{m_1}})\overline{u_1(\sigma^{q^{m_2}})}\rangle_n
  +q^4(q+1)^2\langle{\sigma}^{q^{m_1}-q^{m_2}}\rangle_n\\
&-q^4(q^2-1)\langle u_1(\sigma^{q^{m_1}}){\sigma}^{-q^{m_2}}+\overline{u_1(\sigma^{q^{m_2}})}{\sigma}^{q^{m_1}}\rangle_n.
\end{split}
\end{equation}

Since $0\leq m_1,m_2\leq n-1$, it follows that $m_1-m_2=kn$ for $k\in\Z$ if and only if $m_1=m_2$. By Lemma \ref{lemma_Appendix}(3), we have
\begin{equation} \label{sigma_sigma}                
\left\langle{\sigma}^{q^{m_1}-q^{m_2}}\right\rangle_n=
\left\{                         
\begin{array}{ll}               
1 &~~~~\text{if}~m_1=m_2, \\             
0 &~~~~\text{otherwise}.    
\end{array}                     
\right.                         
\end{equation}
This means that
\begin{equation}\label{ingre-1}
\sum_{0\leq m_1,m_2\leq n-1}\langle \sigma^{q^{m_1}-q^{m_2}}\rangle_n=n.
\end{equation}

Similarly, by Lemma \ref{lemma_Appendix}(3), we have
\begin{equation}\label{u_sigma}
\begin{split}
&{\left\langle u_1(\sigma^{q^{m_1}})\sigma^{-q^{m_2}}\right\rangle_n=
 \sum_{k=0}^{+\infty}\frac{\langle\sigma^{q^{k+m_1}-q^{m_2}}\rangle_n}{q^k}}\\
=&~
\left\{
\begin{array}{ll}
\sum_{k=0}^{+\infty}\frac{1}{q^{n-(m_1-m_2)+kn}}=\frac{q^{m_1-m_2}}{q^n-1} &~~~~\text{if}~m_1>m_2,\\
\sum_{k=0}^{+\infty}\frac{1}{q^{m_2-m_1+kn}}=\frac{q^{n-(m_2-m_1)}}{q^n-1} &~~~~\text{if}~m_1\leq m_2.
\end{array}
\right.
\end{split}
\end{equation}
This means that
\begin{equation}\label{ingre-2}
\begin{split}
& \sum_{0\leq m_1,m_2\leq n-1}\langle u_1(\sigma^{q^{m_1}})\sigma^{-q^{m_2}} \rangle_n
=\sum_{0\leq m_2<m_1\leq n-1}\frac{q^{m_1-m_2}}{q^n-1}
  + \sum_{0\leq m_1\leq m_2\leq n-1}\frac{q^{n-(m_2-m_1)}}{q^n-1}\\
&~ =\frac{n}{q^n-1}(q+q^2+\cdots+q^n)=\frac{nq}{q-1}.
\end{split}
\end{equation}

Moreover, by Lemma \ref{lemma_Appendix}(3), we have
\begin{equation}\label{sigma-2}
\begin{split}
&{\left\langle u_1(\sigma^{q^{m_1}})\overline{u_1(\sigma^{q^{m_2}})}\right\rangle_n=
 \sum_{k_1=0}^{+\infty}\sum_{k_2=0}^{+\infty}\frac{\langle\sigma^{q^{k_1+m_1}-q^{k_2+m_2}}\rangle_n}{q^{k_1+k_2}}}\\
=&~
\left\{
\begin{array}{ll}
(\frac{1}{q^{m_1-m_2}}+\frac{1}{q^{n-(m_1-m_2)}})\frac{q^{2+n}}{(q^2-1)(q^n-1)} &~~~~\text{if}~m_1>m_2,\\
(\frac{1}{q^{m_2-m_1}}+\frac{1}{q^{n-(m_2-m_1)}})\frac{q^{2+n}}{(q^2-1)(q^n-1)} &~~~~\text{if}~m_1\leq m_2.
\end{array}
\right.
\end{split}
\end{equation}
This means that (similar to the reduction process of \eqref{ingre-2})
\begin{equation}\label{ingre-3}
\sum_{0\leq m_1,m_2\leq n-1}\langle u_1(\sigma^{q^{m_1}})\overline{u_1(\sigma^{q^{m_2}})} \rangle_n
=\frac{nq^2}{(q-1)^2}.
\end{equation}

By substituting \eqref{ingre-1}, \eqref{ingre-2} and \eqref{ingre-3} into \eqref{appendix_AA}, we have
\begin{equation}\label{simp-1}
\sum_{0\leq m_1,m_2\leq n-1}\langle A_{m_1}\overline{A}_{m_2}\rangle_n=nq^4.
\end{equation}
By \eqref{appendix_ave} and \eqref{simp-1}, it follows that \eqref{appendix-reduce} holds. The proof of \eqref{appendix} is complete.


\bibliography{References}

\end{document}